\documentclass[12pt]{article}
\usepackage{amsmath,amsfonts,amssymb,amsthm}
\usepackage{color}
\newcommand{\I}{{\bf 1}}
\newtheorem{proposition}{Proposition}[section]
\newtheorem{theorem}[proposition]{Theorem}
\newtheorem{corollary}[proposition]{Corollary}
\newtheorem{lemma}[proposition]{Lemma}
\newtheorem{remark}[proposition]{Remark}

\newtheorem{example}[proposition]{Example}

\numberwithin{equation}{section}

\newcommand{\nc}{\newcommand}
\nc{\bF}{{\mathbf F}}
\nc{\bS}{{\mathbf S}}
\nc{\bN}{{\mathbf N}}
\nc{\bM}{{\mathbf M}}
\nc{\cB}{{\mathcal B}}
\nc{\cG}{{\mathcal G}}
\nc{\cS}{{\mathcal S}}
\nc{\cM}{{\mathcal M}}
\nc{\R}{{\mathbb R}}
\nc{\N}{{\mathbb N}}
\nc{\Z}{{\mathbb Z}}
\nc{\BX}{{\mathbb X}}
\nc{\BG}{{\mathbb G}}
\nc{\BH}{{\mathbb H}}

\DeclareMathOperator{\supp}{supp}

\nc{\BP}{\mathbb{P}}
\nc{\BE}{\mathbb{E}}
\nc{\BQ}{\mathbb{Q}}

\begin{document}
\renewcommand{\thefootnote}{\fnsymbol{footnote}}
\author{G\"unter Last\footnotemark[1]}
\footnotetext[1]{guenter.last@kit.edu, 
Karlsruhe Institute of Technology, Institute for Stochastics, 76131 Karlsruhe, Germany. }
\title{Tail processes and tail measures:\\
An approach via Palm calculus}
\date{\today}
\maketitle

\begin{abstract} 
\noindent 
Using an intrinsic approach, we study some properties
of random fields which appear as tail fields of regularly varying stationary
random fields. The index set is allowed to be a general locally compact Hausdorff 
Abelian group $\BG$. The values are taken in a measurable cone, equipped
with a pseudo norm.  We first discuss
some Palm formulas for the exceedance random measure $\xi$ associated with
a stationary (measurable)  random field $Y=(Y_s)_{s\in \BG}$.
It is important to allow the underlying stationary measure to be $\sigma$-finite.
Then we proceed to a random field (defined on a probability space)
which is spectrally decomposable, in a
sense which is motivated by extreme value theory.
We characterize mass-stationarity of the exceedance random measure in terms
of a suitable version of the classical Mecke equation. 
We also show that the associated stationary measure is homogeneous,
that is a tail measure. We then proceed 
with establishing and studying the
spectral representation of stationary tail measures
and with characterizing a moving shift representation.
Finally we discuss anchoring maps and the candidate extremal index.
\end{abstract}


\vspace{0.2cm}

\section{Introduction}

The tail process of regularly varying time series
was introduced in \cite{BaSe09}.
It is a useful tool for describing and handling
the extreme value behavior of such time series;
see e.g.\ \cite{DoHaSo18,KulSo2020,So21}. The recent
paper \cite{Plan21} has made some interesting connections to Palm
theory for point processes on $\Z^d$. In particular it
has been observed there that the exceedance point process
of the tail process is point-stationary in the sense of \cite{Tho07};
see also \cite{LaTho09}. One aim of the present paper 
is to extend \cite{Plan21} to the case
of a general locally compact Hausdorff Abelian
group $\BG$, for instance $\BG=\R^d$.
Even in the case $\BG=\Z^d$ our approach
will provide further insight into the results from \cite{Plan21}.
Another aim is to extend the concept of a tail measure (as defined in
\cite{DoHaSo18,So21}) to spaces of functions on Abelian groups, to relate
these measures to Palm calculus and to study their spectral
representation.

Section \ref{secPalm} contains some basic definitions and
facts from Palm theory. In Section  \ref{secalloc} 
we will first
provide a modest but useful generalization of 
\cite[Theorem 4.1]{La10} on allocations and Palm measures.
Then we summarize some facts on point- and mass-stationarity.
In Section \ref{secexceedance} we consider
a field $Y=(Y_s)_{s\in \BG}$  indexed by the group. The field 
takes its values in a measurable cone $\BH$ equipped with
a pseudo norm $|\cdot|$. A key example is $\BH=\R^d$
with the Euclidean norm. We require $Y$ to have natural measurability
properties but do not impose continuity or separability
assumptions. The exceedance random measure is defined by 
$\xi:=\int\I\{s\in \cdot,|Y_s|>1\}\lambda(ds)$, where
$\lambda$ is a Haar measure on $\BG$.
We briefly discuss stationarity, mass-stationarity and
the Palm measure of $\xi$.
For our purposes it is important to allow the underlying stationary
measure $\BP$ to be infinite (but $\sigma$-finite).
The Palm measure of $\xi$
is simply the restriction of $\BP$ to the event $\{|Y_0|>1\}$.
Starting with Section \ref{sectailprocess} we shall work
on a suitable canonical function space $(\bF,\mathcal{F})$
with the field $Y$ given as the identity on $\bF$.
At the cost of a more abstract setting,  
this could be  generalized along the lines of  Remark \ref{rgen1}.
In  Section \ref{sectailprocess} we assume that $Y$ 
is spectrally decomposable with index $\alpha>0$ w.r.t.\  
a probability measure $\BQ$ on  $(\bF,\mathcal{F})$.
This assumption
is strongly motivated by \cite{BaSe09} and means
that $|Y_0|$ has a Pareto distribution (on $(1,\infty)$) with parameter $\alpha$ and is 
independent of $W:=(|Y_0|^{-1}Y_s)_{s\in \BG}$.
Our Theorem \ref{tMecketail} shows that  the exceedance random measure
$\xi$ is mass-stationary in the sense of \cite{LaTho09} if and only if
$W$ satisfies the space shift formula \eqref{mecketail},
a version of the classical Mecke equation from \cite{Mecke}.
This generalizes the main result in \cite{Plan21} from $\Z^d$
to general locally compact Hausdorff Abelian groups.
In establishing this result, we will not refer to a regularly varying
field in the background. 
Under the assumptions of Theorem \ref{tMecketail},
general Palm theory essentially guarantees the existence of
a stationary $\sigma$-finite
measure $\nu$ such that $\BQ$ is the Palm measure of $\xi$ w.r.t.\ $\nu$,
that is $\BQ=\nu(\cdot\cap\{|Y_0|>1\})$.
In Section \ref{sectailmeasure} we shall prove among other things that
$\nu$ is $\alpha$-homogeneous, that is a tail measure.
In Section \ref{sectionspectralr} we shall prove 
with Theorem \ref{tspecrepresentation} that any stationary tail
measure $\nu$ has a spectral representation. 
While the existence of such a representation
can be derived from  \cite[Proposition 2.8]{EvansMol18}
(see Remark \ref{r6.34}), our result provides an explicit construction
of the spectral measure in terms of the Palm measure $\BQ$ of $\xi$
along with further properties.
Theorem \ref{tspecrepresentation} extends the stationary case of
\cite[Theorem 2.4]{DoHaSo18}  (dealing with $\BG=\Z$)
and \cite[Theorem 2.3]{So21} (dealing with the case $\BG=\R$)
to general Abelian groups. We also characterize a moving
shift representation. In the final Section \ref{secanchor}
we study anchoring maps, as defined in \cite{Plan21,So21}
for mass-stationary fields with the property $\BQ(0<\xi(\BG)<\infty)=1$.
Proposition \ref{p6.3} extends \cite[Proposition 3.2]{Plan21}
to general Abelian groups. In the remainder of the section
we assume $Y$ to be spectrally decomposable. Motivated
by \cite[Section 2.3]{So21} we provide some information
on the candidate extremal index.

In this paper %suggest to define a 
we treat tail processes
in an intrinsic way, namely as a spectrally decomposable random
field $Y=(Y_s)_{s\in \BG}$ such that $\xi$ is mass-stationary.
This is in line with the developments in \cite{DoHaSo18,KulSo2020,So21}
and in the recent preprints \cite{BlaHashShev21,Hashorva21}.

\section{Some Palm calculus}\label{secPalm}

Assume that $\BG$ is a locally compact Hausdorff  group with Borel $\sigma$-field
$\cG$ and (non-trivial) Haar measure $\lambda$.
Important special cases are $\BG=\Z^d$ with $\lambda$ being the counting measure
and $\BG=\R^d$ with $\lambda$ being the Lebesgue measure.
Let $\bM$ denote the space of measures on $\BG$ which are locally finite
(that is, finite on compact sets) and let $\cM$ be the smallest $\sigma$-field
on $\bM$ making the mappings $\mu\mapsto \mu(B)$
measurable for all $B\subset \BG$. 
Let $\bN$ be the measurable subset of $\bM$ of those $\mu\in\bM$
which are integer-valued on relatively compact Borel sets.
Let $(\Omega,\mathcal{A},\BP)$ be a $\sigma$-finite measure space.
At the moment the reader might think of $\BP$ as of a probability measure.
However,  for our later purposes it is important to allow for $\BP(\Omega)=\infty$.
Still we shall use a probabilistic language.
A {\em random measure} (resp.\ {\em point process}) $\xi$ on $\BG$ is 
a measurable mapping $\xi\colon\Omega\rightarrow\bM$
(resp.\ $\xi\colon\Omega\rightarrow\bN$). We find it convenient
to use this terminology even without reference to a (probability) measure
on $(\Omega,\mathcal{A})$. We often use the kernel notation
$\xi(\omega,B):=\xi(\omega)(B)$, $(\omega,B)\in\Omega\times\mathcal{G}$.
A point process $\xi$ is said to be {\em simple}, if 
$\xi(\omega,\{s\})\le 1$ for all $(\omega,s)\in \Omega\times\BG$.

Next we give a short but self-contained introduction
into Palm calculus, using the setting from \cite{Neveu} and \cite{LaTho09}.
A more comprehensive summary can be found in \cite{La10}.
Assume that $\BG$ acts measurably on $(\Omega,\mathcal{A})$.
This means that there is a  family
of measurable mappings $\theta_s\colon\Omega\to\Omega$, $s\in \BG$,
such that $(\omega,s)\mapsto \theta_s\omega$ is measurable,
$\theta_0$ is the identity on $\Omega$ and
\begin{align}\label{flow}
\theta_s \circ \theta_t =\theta_{s+t},\quad s,t\in \BG,
\end{align}
where $\circ$ denotes composition. The family $\{\theta_s:s\in\BG\}$ 
is said to be (measurable) {\em flow} on $\Omega$.
A random measure on $\BG$ 
is said to be {\em invariant} (w.r.t.\ to the flow) or {\em flow-adapted}  if
\begin{align}\label{adapt}
\xi(\omega,B+s)=\xi(\theta_s\omega,B),\quad \omega\in \Omega,\,s\in \BG,
B\in\cG.
\end{align}
Let us illustrate these concepts with two examples.

\begin{example}\label{ex2.1}\rm Assume that $(\Omega,\mathcal{F})=(\bM,\mathcal{M})$
and define $\theta_s\mu:=\mu(\cdot+s)$, for $\mu\in\bM$ and $s\in\BG$.
Then $\{\theta_s:s\in\BG\}$ is a flow and the identity on $\bM$ is
invariant.
\end{example} 

\begin{example}\label{ex2.2}\rm
Let $\BH$ be a (non-empty) Polish space equipped with the Borel $\sigma$-field $\mathcal{H}$
and consider the space $\BH^\BG$ of all functions $\omega\colon\BG\to\BH$.
For each $s\in \BG$ we define the shift-operator 
$\theta_s\colon \BH^\BG\to \BH^\BG$ by $\theta_s\omega:=\omega(\cdot+s)$.
Assume now that $\bF$ is shift-invariant subset of $\BH^\BG$ equipped with a $\sigma$-field
$\mathcal{F}$ such that $(\omega,s)\mapsto (\theta_s\omega,\omega(0))$ is
measurable with respect to $\mathcal{F}\otimes\mathcal{H}$. For instance we can take
$\BG=\R^d$, $\BH=\R$, $\bF$ as the {\em Skorohod space} of all c\`adl\`ag functions 
(see e.g.\ \cite{Janson20})
and $\mathcal{F}$ as the smallest $\sigma$-field rendering the mappings
$\omega\mapsto \omega(t)$, $t\in\BG$, measurable. 
Then even $(\omega,t)\mapsto \omega(t)$ is measurable and
therefore also $(\omega,s)\mapsto \theta_s\omega$, as required.
An example of an invariant random measure (defined on $\bF$) is
$\xi(\omega):=\int\I\{t\in \cdot\}f(\theta_t\omega)\,\lambda(dt)$,
where $f\colon\bF\to[0,\infty)$ is measurable and bounded. 
\end{example}

In view of the preceding examples it is helpful
to think of $\theta_s\omega$ as of $\omega$ {\em shifted} by $s$.
A  measure $\BP$ on $(\Omega,\mathcal{A})$ is called
{\em stationary} if it is invariant under the flow, i.e.
$$
\BP\circ\theta_s=\BP,\quad s\in \BG,
$$
where $\theta_s$ is interpreted as a mapping from $\mathcal{A}$ to $\mathcal{A}$
in the usual way:
$$
\theta_sA:=\{\theta_s\omega:\omega\in A\},\quad A\in\mathcal{A},\, s\in \BG.
$$
Throughout the paper $\BP$ will denote a  $\sigma$-finite stationary measure 
on $(\Omega,\mathcal{A})$.

Let $B\in\cG$ be a set with positive and finite Haar measure $\lambda(B)$
and $\xi$ be an invariant random measure on $\BG$.
The measure
\begin{align} \label{Palm}
\BP_\xi(A):=\lambda(B)^{-1}\iint \I_A(\theta_s\omega)\I_B(s)\,
\xi(\omega,ds)\,\BP(d\omega), \quad A\in\mathcal{A},
\end{align}
is called the {\em Palm measure} of $\xi$ (with respect to $\BP$).

For discrete groups the previous definition becomes very simple:

\begin{example}\label{ex2.3}\rm Assume that $\BG$ is discrete.
Then we can take $B:=\{0\}$ and obtain that
\begin{align*}
\BP_\xi(A)=\BE\I_A\xi(\{0\}).
\end{align*}
\end{example}

The {\em intensity} of $\xi$ is the number $\gamma_\xi:=\BE[\xi(B)]=\BP_\xi(\Omega)$.
If this intensity  is positive and finite then the normalized Palm measure 
$$
\BP^0_\xi:=\gamma_{\xi}^{-1}\BP_\xi
$$
is called {\em Palm probability measure} of $\xi$ (w.r.t.\ $\BP$).
Note that $\BP_\xi$ and $\BP^0_\xi$
are defined on the underlying space $(\Omega,\mathcal{A})$.
The {\em Palm distribution} of $\xi$ is the distribution $\BP^0_\xi(\xi\in\cdot)$
of $\xi$ under $\BP^0_\xi$. 
If $\xi$ is a {\em simple} point process (that is $\xi(\{s\})\le 1$ for all $s\in \BG$),
the number $\BP_\xi^0(A)$ can be interpreted as the conditional probability of 
$A\in\mathcal{A}$ given that $\xi$ has a point at $0\in \BG$.

In the general case the Palm measure $\BP_\xi$ is $\sigma$-finite. 
Moreover, if $\gamma_\xi>0$ and $A\in\mathcal{A}$ is flow-invariant 
(that is $\theta_sA=A$ for each $s\in\BG$), then
$\BP(A)=0$ iff $\BP_\xi(A)=0$. 
Since the definition \eqref{Palm}  
does not depend on $B$, we have the {\em refined Campbell theorem}
\begin{align}\label{erefinedC}
\iint f(\theta_s\omega,s)\,\xi(\omega,ds)\, \BP(d\omega)=
\iint f(\omega,s)\,\lambda(ds)\,\BP_\xi(d\omega)
\end{align}
for all measurable $f\colon\Omega\times \BG\to [0,\infty]$.
We write this as
\begin{align} \label{21}
\BE\left[\int f(\theta_s,s)\,\xi(ds)\right]=
\BE_{\BP_\xi}\left[\int f(\theta_0,s)\,\lambda(ds)\right],
\end{align}
where $\BE$ and $\BE_{\BP_\xi}$ denote integration with respect to $\BP$
and $\BP_\xi$, respectively. Note that
\begin{align} \label{e430}
\BP_\xi(\xi(\BG)=0)=0.
\end{align}
If $\xi$ is a point process, then $\BP_\xi$ is concentrated on the event
$\{\omega\in\Omega:\xi(\omega,\{0\})\ge 1\}$. 

Let $\xi$ be an invariant random measure on $\BG$ and
let $\tilde h\colon\Omega\times \BG\to[0,\infty]$ be a measurable function such
$\int \tilde h(\theta_0,s)\,\xi(ds)=\I\{\xi(\BG)>0\}$ $\BP$-a.e.
Then the refined Campbell theorem implies the
{\em inversion formula}
\begin{align}\label{einversion}
\BE\I\{\xi(\BG)>0\}f=\BE_{\BP_\xi}\int f(\theta_{-s})\tilde{h}(\theta_{-s},s)\,\lambda(ds),
\end{align}
for each measurable $f\colon \Omega\to[0,\infty]$;
see also \cite{Mecke}.
This shows that the
restriction of the measure $\BP$ to $\{\xi(\BG)>0\}$
is uniquely determined by $\BQ$.

Let $\xi$ and $\eta$ be two invariant random measures on $\BG$
and $g\colon \Omega\times \BG\to[0,\infty]$ be measurable. 
Neveu's \cite{Neveu} {\em exchange formula} says that
\begin{align}\label{neveu0}
\BE_{\BP_\xi}\left[\int g(\theta_0,s)\,\eta(ds)\right]
=\BE_{\BP_\eta}\left[\int g(\theta_s,-s)\,\xi(ds)\right].
\end{align}

\section{Allocations, point- and mass-stationarity}\label{secalloc}

As in Section \ref{secPalm} we consider a measurable space $(\Omega,\mathcal{A})$
equipped with a measurable flow $\{\theta_s:s\in\BG\}$ and a stationary
$\sigma$-finite measure $\BP$.

A measurable function $\tau\colon \Omega\times \BG\to \BG\cup\{\infty\}$
(it is understood here that $\infty\notin \BG$) 
is said to be an {\em allocation}, if it satisfies
the covariance property
\begin{align}\label{ecovariance}
\tau(\theta_t\omega,s-t)=\tau(\omega,s)-t,\quad s,t\in \BG,\,\omega\in\Omega, 
\end{align}
where $\infty-t:=\infty$.
Given such an allocation we define the (random) sets
\begin{align}
C_\tau(s):=\{t\in \BG:\tau(t)=s\},\quad s\in \BG,
\end{align}
where, as usual, $\tau(t):= \tau(\cdot,t)$.

The following result generalizes \cite[Theorem 4.1]{La10}.
The latter arises in the special case where $\BG=\R^d$ and
$\xi$ equals Lebesgue measure. We denote by $\supp\mu$ the {\em support}
of a measure $\mu$ on $\BG$.

\begin{proposition}\label{palloc} Suppose that $\xi$ is an invariant random
measure and that $\eta$ is a simple invariant point process.
Let $\tau$ be an allocation satisfying
\begin{align*}
\tau(s)\in \supp\eta\cup\{\infty\},\quad \xi\text{-a.e.\ $s\in \BG$},\,\BP\text{-a.e.}
\end{align*}
Let $h\colon \Omega\times\Omega\to[0,\infty]$ be
measurable. Then
\begin{align}\label{ealloc}
\BE_{\BP_\xi} \I\{\tau(0)\ne\infty\}h(\theta_0,\theta_{\tau(0)})
=\BE_{\BP_\eta} \int_{C_\tau(0)} h(\theta_s,\theta_0)\,\xi(ds).
\end{align}
\end{proposition}
\begin{proof} It follows from \eqref{adapt}  and \eqref{ecovariance}
that the event consisting of all $\omega\in\Omega$ satisfying
\begin{align*}
\xi(\omega,\{s\in\BG:\tau(\omega,s)\notin (\supp\eta(\omega)\cup\{\infty\}\})=0
\end{align*}
is shift-invariant. Therefore,
\begin{align}\label{e2.7}
\BP_\xi(\tau(0)\notin (\supp\eta \cup\{\infty\}))=0.
\end{align}

The proof proceeds now as the one of \cite[Theorem 4.1]{La10}, applying
the exchange formula \eqref{neveu0} 
instead of the refined Campbell theorem.
We apply \eqref{neveu0} 
with the function $(\omega,s)\mapsto h(\omega,\theta_s\omega)\I\{\tau(\omega,0)=s\}$. 
In view of \eqref{e2.7}, the left-hand coincides with the left-hand side of
\eqref{ealloc}. The right-hand side equals
\begin{align*}
\BE_{\BP_\eta} \int h(\theta_s,\theta_0)\I\{\tau(\theta_s,0)=-s\}\,\xi(ds).
\end{align*}
Since $\tau(\theta_s,0)=\tau(\theta_0,s)-s$, this equals the
right-hand side of \eqref{ealloc}.
\end{proof}

Let $\xi$, $\eta$ and $\tau$ be as in Proposition \ref{palloc}
and assume moreover that $\BP$-a.e.\ 
$\xi(C_\tau(s))=1$ for all $s\in\supp\eta$. 
Then \eqref{ealloc} implies the {\em shift-coupling}
\begin{align}\label{eshiftcoupling}
\BE_{\BP_\xi} \I\{\tau(0)\ne\infty\}\I\{\theta_{\tau(0)}\in\cdot)=\BP_\eta. 
\end{align}
The additional assumption on $\tau$ is equivalent to the {\em balancing} property
\begin{align}\label{ebalance}
\int \I\{\tau(s)\ne\infty,\tau(s)\in\cdot\}\,\xi(ds)=\eta, \quad \BP\text{-a.e.}
\end{align}
Since the above balancing event is easily seen to be flow-invariant,
equation \eqref{ebalance} does also hold $\BP_\xi$-a.e.\  
and $\BP_\eta$-a.e. Of particular interest is the case $\xi=\eta$.
Then \eqref{ebalance} implies  $\BP$-a.e.\
that $\tau(s)\ne\infty$ for all $s\in\supp\xi$ and \eqref{ebalance} 
means that $\tau(\omega,\cdot)$ induces for $\BP$-a.e.\ $\omega$ a
bijection between the points of $\supp\xi$. We say that
$\tau$ is a {\em bijective point map} for $\xi$ w.r.t.\ $\BP$ (see \cite{Tho07,He:La:05})
and use this terminology also for other measures $\BP$.

Given an invariant simple point process $\xi$ and a measure $\BQ$  
on $\Omega$, we call $\xi$ {\em point-stationary} if $\BQ(0\notin\supp\xi)=0$
and $\BQ(\theta_{\tau(0)}\in\cdot)=\BQ$ holds for each
bijective point map $\tau$ for $\xi$ w.r.t.\ $\BQ$.
It was proved in \cite{He:La:05} that a $\sigma$-finite measure 
$\BQ$  on $\Omega$ is point-stationary iff it is the Palm measure
of $\xi$ with respect to some $\sigma$-finite stationary measure
on $\Omega$. A key ingredient of the proof is the following
intrinsic  characterization of general Palm measures; see  \cite[Satz 2.5]{Mecke}.
Mecke proved his fundamental result in a canonical setting.
As  discussed in \cite{LaTho09} his proof applies in our more general framework.

\begin{theorem}[Mecke 1967]\label{tMecke} Let $\xi$ be an invariant
random measure on $\BG$ and $\BQ$ be a $\sigma$-finite measure
on $(\Omega,\mathcal{A})$. Then $\BQ$ is the Palm measure  of $\xi$
w.r.t.\ a $\sigma$-finite stationary measure on $\Omega$ iff $\BQ(\xi(\BG)=0)=0$ and
\begin{align}\label{mecke1}
\BE_\BQ \int g(\theta_s,-s) \,\xi(ds)=
\BE_\BQ \int g(\theta_0,s) \,\xi(ds)
\end{align}
for all measurable $g\colon\Omega\times \BG\to [0,\infty]$.
Equation \eqref{mecke1} determines the stationary measure
on   $\{\xi(\BG)>0\}$.
\end{theorem}

The final assertion of Theorem \ref{tMecke} follows from
the inversion formula \eqref{einversion}.
.

Point stationarity was extended in \cite{LaTho09} to 
{\em mass-stationarity}  of an invariant random measure $\xi$ w.r.t.\ a given
$\sigma$-finite measure $\BQ$ on $\Omega$. 
Roughly speaking, mass-stationarity of $\xi$ can be described
as follows. Let $C\in\mathcal{G}$ be a relatively compact set
with positive Haar measure whose boundary is not charged by $\lambda$. 
Let $U$ be a random element of $\BG$, independent of $\xi$
and with distribution $\lambda(C)^{-1}\lambda(C\cap\cdot)$. 
Given $(\xi,U)$ pick a random point $V$ according to the normalized
restriction of $\lambda$ to $C-U$. If $\BQ$ is the Palm measure
of $\xi$ w.r.t.\ some stationary measure, then
\begin{align*}
\BQ((\xi\circ\theta_V,U+V)\in\cdot)=\BQ((\xi,U)\in\cdot).
\end{align*}
As shown  by \cite[Theorem 6.3]{LaTho09},  a version of this property 
(assumed to be true for all $C$ as above) is equivalent
to \eqref{mecke1} and hence provides another intrinsic chracterization
of Palm measures. 
Justified by this result we call $\xi$ {\em mass-stationary} (w.r.t.\ $\BQ$) 
if \eqref{mecke1} holds.
In this paper $\BQ$ will always denote a probability measure
while, as a rule, $\BP$ is only $\sigma$-finite.

%Note that the intensity of $\xi$ (w.r.t.\ $\BP$) equals
%$\BP_\xi(\Omega)=1$.

%\section{Allocations}\label{secalloc}

%We let $(\Omega,\mathcal{A})$ be as in Section \ref{secPalm}
%and suppose that $\BP$ is a stationary measure on $\Omega$.

\section{Exceedance random measures}\label{secexceedance}

Let $\BH$ be a (non-empty) Polish space equipped with the Borel $\sigma$-field $\mathcal{H}$.
Assume that $|\cdot|\colon \BH\to [0,\infty)$ is a measurable mapping.
One might think of $\BH=\R^d$ equipped with the Euclidean norm.

In this and later sections we consider a measurable mapping
$Y\colon \Omega\times \BG\to\BH$. For $s\in \BG$ we write $Y_s$
for the random variable $\omega\mapsto Y_s(\omega)$.
Then $Y$ can be considered as a (measurable) random field $(Y_s)_{s\in \BG}$.
We assume the shift-covariance
\begin{align}\label{e123}
Y_s(\theta_t\omega)=Y_{s+t}(\omega),\quad (\omega,s,t)\in\Omega\times \BG\times \BG.
\end{align}
We call
\begin{align}\label{exceedancemeasure}
\xi:=\int \I\{s\in\cdot ,|Y_s|>1\}\,\lambda(ds)
\end{align}
the {\em exceedance measure} of $Y$. 
By \eqref{e123} this is an invariant random measure.
If $\BG$ is discrete, $\xi$ is a simple point process.
The Palm measure of $\xi$ takes a rather simple form:

\begin{lemma}\label{lsimplePalm} Let $\BP$ be
a $\sigma$-finite stationary measure on $\Omega$. Then
\begin{align}
\BP_\xi=\BP(\cdot\cap \{|Y_0|>1\}).
\end{align}
\end{lemma}
\begin{proof} Let $B\in\mathcal{G}$ satisfy $\lambda(B)=1$. Then
\begin{align*}
\BP_\xi&=\BE_\BP \int \I\{\theta_s\in\cdot\}\I\{s\in B\}\,\xi(ds)
=\BE_\BP \int \I\{\theta_s\in\cdot\}\I\{s\in B,|Y_s|>1\}\,\lambda(ds)\\
&=\int_B\BE_\BP[\I\{\theta_s\in\cdot,|Y_s|>1\}]\,\lambda(ds)
=\int_B\BE_\BP[\I\{\theta_0\in\cdot,|Y_0|>1\}]\,\lambda(ds),
\end{align*}
where we have used stationarity of $\BP$ to get the final equation.
This proves the assertion.
\end{proof}

In this paper we will mostly be concerned
with a probability measure $\BQ$ on $(\Omega,\mathcal{A})$ 
such that $\xi$ is mass-stationary w.r.t.\ $\BQ$.
In this case Theorem \ref{tMecke} shows that
there exists a unique $\sigma$-finite stationary measure
$\BP$ on $\Omega$ satisfying $\BP_\xi=\BQ$ and $\BP(\xi(\BG)=0)=0$.
If $\BQ(\xi(\BG)<\infty)=1$ then  $\BP$ cannot be finite:

\begin{remark}\rm Let $\BQ$ be as above and
assume that $\BQ(0<\xi(\BG)<\infty)=1$.
Define a measure $\BP$ on $\Omega$ by
\begin{align}
\BP:=\BE_\BQ \xi(\BG)^{-1}\int \I\{\theta_s\in \cdot\}\,\lambda(ds).
\end{align}
Since $\xi(\BG)$ is invariant under shifts, this measure is stationary.
Let $f\colon\bM\to[0,\infty]$ be measurable and $B\in\cG$ with
$\lambda(B)=1$. A simple calculation (using invariance of $\xi$
and Fubini's theorem) shows that
\begin{align*}
\BE_\BP \int \I\{s\in B\}f\circ\theta_s\,\xi(ds)
=\BE_\BQ \xi(\BG)^{-1}\int f\circ\theta_s\,\xi(ds).
\end{align*}
Since we have assumed $\xi$ to be mass-stationary,
we can use the Mecke equation \eqref{mecke1} to find that
the latter expression equals $\BE_\BQ f$.
Hence %$\xi$ has intensity $1$ (w.r.t.\ $\BP$) and 
$\BQ$ is the Palm measure of $\xi$.
Note that $\BP(\Omega)=\infty$, unless $\lambda(\BG)<\infty$.
\end{remark}

%\begin{proposition}\label{palloc2}
%Let $\tau$ and $\eta$ satisfy the assumptions of
%Proposition \ref{palloc} and let
%$h\colon \Omega\times\Omega\to[0,\infty]$ be
%measurable. Then
%\begin{align}\label{ealloc2}
%\BE_\BQ \I\{\tau(0)\ne\infty\}h(\theta_0,\theta_{\tau(0)})
%=\BE_{\BP_\eta} \int_{C(0)} h(\theta_s,\theta_0)\,\xi(ds).
%\end{align}
%\end{proposition}
%\begin{proof} The result follows from Proposition \ref{palloc}.
%\end{proof}

\section{Spectrally decomposable fields}\label{sectailprocess}

In this and later sections we take $(\Omega,\mathcal{A})$ as the  function space 
$(\bF,\mathcal{F})$ satisfying the assumptions of Example \ref{ex2.2}.
In addition we assume that $\BH$ is a {\em measurable cone}, that
is, there exists a measurable mapping $(u,x)\mapsto u\cdot x$
from $(0,\infty)\times\BH$ to $\BH$ such that $1\cdot x=x$ and
$u\cdot(v\cdot x)=(uv)\cdot x$ for all $x\in\BH$ and $u,v\in(0,\infty)$.
We mostly write $ux$ instead of $u\cdot x$. 
The function $|\cdot|$ is assumed to be homogeneous,
that is $|ux|=u|x|$ for all $u>0$ and $x\in\BH$.
If $\omega\in\BH^\BG$ and $u\in(0,\infty)$, then, as usual, 
$u\cdot\omega\equiv u\omega$ is the function in
$\BH^\BG$ given by $u\omega(s):=u\cdot \omega(s)$, $s\in\BG$.
%Slightly abusing our notation
%we let for each $s\in \BG$ the shift-operator 
%$\theta_s\colon \BH^\BG\to \BH^\BG$ be defined by $\theta_s\omega:=\omega(\cdot+s)$,
%$\omega\in\BH^\BG$. 
We assume that $\bF$ is  closed under the action of $(0,\infty)$. The 
$\sigma$-field $\mathcal{F}$ has been assumed to render the mapping
$(\omega,s)\mapsto (\theta_s\omega,\omega(0))$ to be measurable and
we assume now in addition that
the mapping  $(\omega,u)\mapsto u\cdot\omega$
is measurable on $\bF\times(0,\infty)$.
If $\BG$ is discrete, then we take $\bF=\BH^\BG$ and equip
it with the product $\sigma$-algebra.
We write $Y_s$ for the mapping $\omega\mapsto \omega(s)$,
$s\in \BG$, and note that $(\omega,s)\mapsto Y_s(\omega)$ is measurable.
We also write $Y:=(Y_s)_{s\in \BG}$, which is simply the identity on $\bF$.
We define another random field $W$, by $W_s:=|Y_0|^{-1}Y_s$ if $|Y_0|>0$ and by
$W_s:= x_0$ otherwise, where $x_0$ is some fixed element of $\BH$ with
$|x_0|=1$. Since we do not assume $\BH$ to contain a zero element
we make the general convention $|y|^{-1}x:=x_0$ whenever
$x,y\in\BH$ and $|y|=0$. 

\begin{remark}\rm Assume that $\bF'\subset\BH^\BG$ is shift-invariant
and closed under the action of $(0,\infty)$. Assume that
$\bF'$ is equipped with the Kolmogorov product
$\sigma$-field, that is, the smallest $\sigma$-field making
the mappings $\omega\mapsto \omega(s)$ (from $\bF'$ to $\BH$) measurable for each $s\in\BG$. 
Assume, moreover, that
$(\omega,s)\mapsto \omega(s)$ is a (jointly) measurable mapping
on $\bF'\times\BG$. Then it is easy to see that
$(\omega,s,u)\mapsto(\theta_s\omega,u\cdot\omega)$ is a measurable
function on $\bF'\times\BG\times(0,\infty)$. This shows that
for a proper choice of $\bF$, the product $\sigma$-field is a natural candidate for $\mathcal{F}$.
%In particular we make this choice if $\BG$ is discrete. In that case
%we can simply take $\bF':=\bF$.
\end{remark}

We often consider a probability measure
$\BQ$ on $\bF$ with the following properties.
The probability measure $\BQ(|Y_0|\in \cdot)$ is a Pareto distribution
on $(1,\infty)$ with parameter $\alpha>0$
and $W$ is independent of $|Y_0|$. To achieve this, 
we take a probability measure $\BQ'$ on $\bF$ such that
$\BQ'(|Y_0|=1)=1$ 
%and \begin{align}\label{epos}
%\int \I\{|Y_s|>0\}\,\lambda(ds)>0\quad \BQ'\text{-a.s.}
%\end{align}
and define
\begin{align}\label{etailprocess}
\BQ:=\iint \I\{u\omega\in\cdot, u>1\}\,\alpha u^{-\alpha-1}\,\BQ'(d\omega)\,du.
\end{align}
For the special groups $\BG=\Z^d$ and $\BG=\R$ such processes
occur in extreme value theory; see the seminal paper \cite{BaSe09} 
(treating $\BG=\Z$) and \cite{KulSo2020,DoHaSo18,So21}. 
%In such a context, $W$ is called {\em spectral process} of $Y$.
Note that $W$ is a measurable function of $Y$, that
$\BQ(W\in\cdot)=\BQ'$ and that the pair $(W,Y_0)$
has the desired properties. We say that $Y$ is 
{\em spectrally decomposable with index $\alpha$} (w.r.t.\ $\BQ$)
or, synonomously, that $\BQ$ is spectrally decomposable.

Define the exceedance random measure $\xi$ by \eqref{exceedancemeasure}.
If $\BQ$ is given as in \eqref{etailprocess}, it is natural to characterize
mass-stationarity of $\xi$ (w.r.t.\ $\BQ$) in terms
of suitable invariance properties of the field $W$.
In the context of tail processes the following property
\eqref{mecketail} was proved in \cite{BaSe09} in the case $\BG=\Z$ 
(see also \cite{DoHaSo18,PlanSo218}) and in the case $\BG=\R$
in \cite{So21}. 
The fact that \eqref{mecketail} implies mass-stationarity 
in the case $\BG=\Z^d$ was derived in \cite{Plan21}, exploiting the connection to
regularly varying random fields. We use here an intrinsic 
non-asymptotic approach. It is worth noticing that \cite{Jan19}
identifies  \eqref{mecketail} (in the case $\BG=\Z$) as being characteristic
for the tail processes introduced in \cite{BaSe09}.

\begin{theorem}\label{tMecketail} 
Assume that $Y$ is  spectrally decomposable with index $\alpha$.
Then the exceedance random measure $\xi$ is mass-stationary if
and only if
\begin{align}\label{mecketail}
\BE_{\BQ} \int g(\theta_sW,-s)\I\{|W_s|> 0\} \,\lambda(ds)=
\BE_{\BQ} \int g(|W_s|^{-1}W,s)|W_s|^\alpha \,\lambda(ds),
\end{align}
holds for all measurable $g\colon\bF\times \BG\to[0,\infty]$. 
\end{theorem}
\begin{proof} Let us first assume that $\xi$ is mass-stationary.
We generalize the arguments from the proof of Lemma 2.2 in \cite{PlanSo218}.
Let $h\colon\bF\times \BG\to[0,\infty]$ be measurable
and $\varepsilon\in(0,1]$. Then
\begin{align*}
I:=\BE_{\BQ} &\int h(\theta_sY,-s)\I\{|Y_s|>\varepsilon\} \,\lambda(ds)\\
&=\BE_{\BQ} \iint h(u\theta_sW,-s)\I\{u|W_s|>\varepsilon\} \I\{u>1\}
\alpha u^{-\alpha-1}\,du \,\lambda(ds)\\
&=\varepsilon^{-\alpha}\,\BE_{\BQ} \iint h(\varepsilon v\theta_sW,-s)\I\{v|W_s|>1\} 
\I\{v>1/\varepsilon\} \alpha v^{-\alpha-1}\,dv \,\lambda(ds),
\end{align*}
where we have made the change of variables $v:=u/\varepsilon$ to
get the second identity. Since $1/\varepsilon >1$ we obtain that
\begin{align*}
I&=\varepsilon^{-\alpha}\,\BE_{\BQ} \int h(\varepsilon \theta_sY,-s)\I\{|Y_s|>1\} 
\I\{|Y_0|>1/\varepsilon\} \,\lambda(ds)\\
&=\varepsilon^{-\alpha}\,\BE_{\BQ} \int h(\varepsilon \theta_sY,-s) 
\I\{|Y_0|>1/\varepsilon\} \,\xi(ds).
\end{align*}
Using now the assumption \eqref{mecke1} together with 
$Y_0=(Y\circ\theta_s)_{-s}$ we arrive at
\begin{align*}
I&=\varepsilon^{-\alpha}\,\BE_{\BQ} \int h(\varepsilon Y,s) 
\I\{|Y_s|>1/\varepsilon\} \,\xi(ds),
\end{align*}
that is 
\begin{align}\label{e5.3}
\BE_{\BQ} &\int h(\theta_sY,-s)\I\{|Y_s|>\varepsilon\} \,\lambda(ds)
=\varepsilon^{-\alpha}\,\BE_{\BQ} \int h(\varepsilon Y,s) 
\I\{|Y_s|>1/\varepsilon\} \,\lambda(ds).
\end{align}
We apply this with $h(Y,s):=g(|Y_s|^{-1}Y,s)$ %(setting $0^{-1}:=0$)
for some measurable function $g\colon\bF\times \BG\to[0,\infty]$,
noting that
$$
h(\theta_sY,-s)=g(|(\theta_sY)_{-s}|^{-1}\theta_sY,-s)
=g(|Y_0|^{-1}\theta_sY,-s)=g(\theta_sW,-s).
$$
Using monotone convergence this yields
\begin{align*}
\BE_{\BQ} \int g(\theta_sW,-s)&\I\{|W_s|> 0\} \,\lambda(ds)\\
&=\lim_{\varepsilon\to 0}\varepsilon^{-\alpha}\,\BE_{\BQ} \int g(|W_s|^{-1}W,s) 
\I\{|Y_0||W_s|>1/\varepsilon\} \,\lambda(ds).
\end{align*}
We have that
\begin{align*}
\varepsilon^{-\alpha}\,\BE_{\BQ} &\int g(|W_s|^{-1}W,s) 
\I\{|Y_0||W_s|>1/\varepsilon\} \,\lambda(ds)\\
&=\varepsilon^{-\alpha}\,\BE_{\BQ} \int g(|W_s|^{-1}W,s) 
\I\{u|W_s|>1/\varepsilon\}\I\{u>1\}\alpha u^{-\alpha-1}\,du \,\lambda(ds)\\
&=\varepsilon^{-\alpha}\,\BE_{\BQ} \int g(|W_s|^{-1}W,s) 
\min\{|W_s|^\alpha\varepsilon^\alpha,1\} \,\lambda(ds)\\
&=\BE_{\BQ} \int g(|W_s|^{-1}W,s) 
\min\{|W_s|^\alpha,\varepsilon^{-\alpha}\} \,\lambda(ds).
\end{align*}
As $\varepsilon\to 0$ the latter term tends to
$\BE_{\BQ} \int g(|W_s|^{-1}W,s) |W_s|^\alpha\,\lambda(ds)$,
yielding \eqref{mecketail}.

To prove the converse implication we assume that
\eqref{mecketail} holds. We take a measurable 
$g\colon \bF\times \BG\to[0,\infty]$ and aim at
establishing \eqref{mecke1}. We have that
\begin{align*}
I':=\BE_{\BQ} \int & g(\theta_sY,-s)\,\xi(ds)
=\BE_{\BQ} \int g(\theta_sY,-s)\I\{|Y_s|>1\}\,\lambda(ds)\\
&=\int \bigg[\BE_{\BQ} \int g(u\theta_sW,-s)\I\{u|W_s|>1\}\,\lambda(ds)\bigg]
\I\{u>1\}\alpha u^{-\alpha-1}\,du.
\end{align*}
For each $u>1$ we can apply \eqref{mecketail} with the function
$\tilde{h}(\omega,s)=g(u\omega,s)\I\{u|\omega(0)|>1\}$. Then
$\tilde{h}(\theta_sW,-s)=g(u\theta_sW,-s)\I\{u|W_s|>1\}$
and 
$$
\tilde{h}(|W_s|^{-1}W,s)=g(u |W_s|^{-1}W,s)\I\{u|W_s|^{-1}>1\}
$$
Therefore
\begin{align*}
I'=\BE_{\BQ} \iint g(u |W_s|^{-1}W,s)|W_s|^{\alpha}\I\{u|W_s|^{-1}>1, u>1\}
\,\alpha u^{-\alpha-1}\,du\,\lambda(ds).
\end{align*}
In the above inner integral we can assume that $|W_s|>0$.
After the change of variables $v:=|W_s|^{-1}u$ we obtain that
\begin{align*}
I'&=\BE_{\BQ} \iint g(v W,s)\I\{v>1, v |W_s|>1\}
\,\alpha v^{-\alpha-1}\,dv\,\lambda(ds)\\
&=\BE_{\BQ} \iint g(Y,s)\I\{|Y_s|>1\}\,\lambda(ds),
\end{align*}
establishing \eqref{mecke1}.
\end{proof}

\begin{remark}\rm The equations \eqref{mecketail} are clearly equivalent
to 
\begin{align}\label{mecketail9}
\BE_{\BQ} h(\theta_{-s}W)\I\{|W_{-s}|> 0\}
=\BE_{\BQ} h(|W_s|^{-1}W)|W_s|^\alpha,\quad \lambda\text{-a.e.\ $s\in\BG$},
\end{align}
for all measurable $h\colon\bF\to[0,\infty]$.
They are also equivalent to the equations
\begin{align}\label{mecketailb}
\BE_{\BQ} \int g(W,-s)\I\{|W_s|> 0\} \,\lambda(ds)
=\BE_{\BQ} \int g(|W_s|^{-1}\theta_sW,s)|W_s|^\alpha \,\lambda(ds).
\end{align}
%If $\BQ(|W_0|=1)=1$ (as under the assumptions of Theorem \ref{tMecketail}) 
as well as to the equations
\begin{align}\label{mecketailc}
\BE_{\BQ} \int g(W,-s)|W_s|^\alpha \,\lambda(ds)
=\BE_{\BQ} \int g(|W_s|^{-1}\theta_sW,s)\I\{|W_s|>0\} \,\lambda(ds).
\end{align}
To see the latter equivalence, we can use the function
$\tilde{h}\colon\bF\times\BG\to [0,\infty)$ given by
$\tilde{h}(\omega,s):=|\omega(-s)|$.
If   $|W_s|>0$ we have that 
$\tilde{h}(|W_s|^{-1}\theta_sW,s)=|W_s|^{-1}$. 
Applying \eqref{mecketailb} with $g\cdot\tilde{h}^\alpha$ instead of $g$
yields \eqref{mecketailc}.
\end{remark}

%For the convenience of the reader we also restate
%Theorem \ref{tMecketail} in the present simpler
%setting. In the remainder of this section we establish the canonical 
%setting $\bF:=\BH^\BG$ and take $\mathcal{F}$ as the product $\sigma$-field.

\begin{remark}\label{tMecketail2}\rm Assume that $\BG$ is discrete
and that $Y$ is  spectrally decomposable  with index $\alpha$. 
Then $\xi$ is mass-stationary iff
\begin{align}\label{mecketail3}
\BE_{\BQ} g(\theta_{-s}W)\I\{|W_{-s}|> 0\}
=\BE_{\BQ} g(|W_s|^{-1}W)|W_s|^\alpha 
\end{align}
holds for all measurable $g\colon\bF \to[0,\infty]$ and all $s\in \BG$. 
\end{remark}

In the case $\BG=\Z$ equation \eqref{mecketail} (see also equation \eqref{mecketail3})
was called {\em time change formula}. In our general setting (and in particular
for $\BG=\Z^d$ or $\BG=\R^d$) this terminology might be replaced
by {\em space shift formula}. We can rewrite \eqref{mecketail} as 
\begin{align}\label{mecketail2}
\BE_{\BQ} \int g(\theta_sW,-s)\,\xi'(ds)=
\BE_{\BQ} \int g(|W_s|^{-1}W,s)|W_s|^\alpha \,\xi'(ds),
\end{align}
where $\xi'$ is the invariant random measure defined by
\begin{align}\label{xiprime}
\xi':=\int\I\{s\in\cdot,|W_s|> 0\}\,\lambda(ds).
\end{align}
This makes the intimate relationship between
\eqref{mecke1} and \eqref{mecketail} even more transparent.

\begin{remark}\label{r5.5}\rm If $\BP(\xi'(\BG)=0)=1$ then 
the equations \eqref{mecketail} are empty. In the spectrally positive case it is, however, 
quite natural to assume that $\BP(\xi'(\BG)=0)=0$.
Indeed, if $\BG$ is discrete or if $Y$ has suitable continuity properties.
then this follows from $\BP(|Y_0|>0)=0$.
\end{remark}

In the spectrally decomposable case the space-shift formula
has the following equivalent version; see \cite{PlanSo218,So21}.

\begin{lemma}\label{l5.3} Assume that $Y$ is  spectrally decomposable with index $\alpha$.
Then the equations \eqref{mecketail} hold iff 
the following equations holds for all measurable $g\colon\bF\times \BG\to[0,\infty]$:
\begin{align}\label{mecke7}
\BE_{\BQ} \int g(Y,s)\I\{|Y_s|> r\}\,\lambda(ds)
=r^{-\alpha}\,\BE_{\BQ}\int g(r\theta_{-s}Y,s)\I\{r|Y_{-s}|> 1\}\,\lambda(ds) ,\quad r>0,
\end{align}
\end{lemma}
\begin{proof} Assume that the equations \eqref{mecke7} hold.
Clearly they are equivalent with \eqref{e5.3}. We have already seen  in the
proof of Theorem \ref{tMecketail} that \eqref{e5.3} implies \eqref{mecketail}.

Assume, conversely, that \eqref{mecketail} holds. We can assume that
$\BP(\xi'(\BG)>0)>0$. (Otherwise there is nothing to prove.)
Define $\tilde\BQ:=\BQ(\cdot\mid\xi'(\BG)>0)$ and 
$\BQ'':=\BQ'(\cdot\mid\xi'(\BG)>0)$. Then  \eqref{etailprocess}
holds with $(\BQ,\BQ')$ replaced by $(\tilde\BQ,\BQ'')$.
Hence $\tilde\BQ$ is spectrally decomposable.
The measure $\BQ$ satisfies \eqref{mecke7} (resp.\ \eqref{mecketail})
iff this is the case for $\tilde\BQ$. Hence it is no loss of generality
to assume that $\BQ(\xi'(\BG)>0)=1$. Corollary \ref{cMecketail}
will show that there is a $\sigma$-finite stationary measure $\nu$ on $\bF$
such that $\BQ=\nu_\xi$. Equation \eqref{mecke7} then follows easily from the homogeneity
of $\nu$, to be discussed in the next section; see Remark \ref{r6.14}. 
\end{proof}

%We proceed with a version of the inversion formula \eqref{einversion}.
In the following we denote $\sigma$-finite (stationary) 
measures on $\bF$ with greek letters. This is at odds with 
Section \ref{secPalm} (and parts of point process literature), 
but in accordance with extreme value theory.

\begin{remark}\label{pinversepalm}\rm 
Assume that $Y$ is  spectrally decomposable 
and that the exceedance random measure $\xi$ is mass-stationary
and satisfies $\BQ (\xi(\BG)=0)=0$.
By Theorem \ref{tMecke}, there exists
a unique $\sigma$-finite stationary measure $\nu$
such that $\nu(\xi(\BG)=0)=0$ and $\nu_\xi=\BQ$. Let $H\colon\bF\times \BG\to[0,\infty)$
be measurable such that 
\begin{align}\label{e54}
\int H(Y,s)\I\{|Y_s|>1\}\,\lambda(ds)=1,\quad \nu\text{-a.e.}
\end{align}
By the inversion formula \eqref{einversion} we have that 
\begin{align*}
\nu=\BE_\BQ\int \I\{\theta_{-s}Y\in\cdot\}H(\theta_{-s}Y,s)\,\lambda(ds).
\end{align*}
Inserting here the spectral decomposition \eqref{etailprocess}, yields 
\begin{align}\label{einverse}
\nu=\BE_\BQ\iint \I\{u\theta_{-s}W\in\cdot\}H(u\theta_{-s}W,s)\I\{u>1\}\,\alpha u^{-\alpha-1}\,du\,\lambda(ds).
\end{align}
\end{remark}

%\begin{proof} Assumption \eqref{e54} implies that $\BQ(\xi(\BG)>0)=1$.
%As recalled in Section \ref{secPalm}, mass-stationarity 
%hence implies the existence of a $\sigma$-finite stationary measure $\nu$
%such that $\nu_\xi=\BQ$.
%By the inversion formula \eqref{einversion} we have that 
%\begin{align*}
%\nu(\cdot\cap \{\xi(\BG)>0\})=\BE_\BQ\int \I\{\theta_{-s}Y\in\cdot\}H(\theta_{-s}Y,s)\,\lambda(ds).
%\end{align*}
%In particular the above right-hand side vanishes on the event $\{\xi(\BG)=0\}$.
%\end{proof}

%The inversion formula \eqref{einverse} can be simplified as follows.

\begin{example}%\label{pinversepalm5}
\rm Consider the setting of Remark \ref{pinversepalm}
and assume moreover that $\BG$ is discrete. 
Let $\tau$ be an allocation such that
\begin{align}\label{e6.2}
\sum_{s\in \BG}\I\{\tau(Y,0)=s,|Y_s|>1\}=1,\quad \nu\text{-a.e.}
\end{align}
Then we can apply \eqref{einverse}
with $H(Y,s):=\I\{\tau(Y,0)=s\}$.
Since $\tau(\theta_{-s}Y,0)=\tau(Y,-s)+s$
we can change variables $s:=-s$ to obtain that the measure \eqref{einverse}
is given by
\begin{align}\label{einverse3}
  \nu=\BE_\BQ\sum_{s\in \BG}\int \I\{u\theta_{s}W\in\cdot,\tau(uW,s)=0\}\I\{u>1\}\,\alpha u^{-\alpha-1}\,du.
\end{align}
\end{example}

\begin{remark}\label{rgen1}\rm The preceding results can be generalized as follows.
Let $(\Omega,\mathcal{A})$ be a measurable space and suppose that 
that there is measurable action $(u,\omega)\mapsto u\omega$ 
from $(0,\infty)\times\Omega$ to $\Omega$. Let $Y$ be a random element
of $\bF$ satisfying \eqref{e123} and also
\begin{align}
Y_s(u\omega)=uY_s(\omega),\quad (\omega,s,u)\in\Omega\times \BG\times (0,\infty).
\end{align}
Let $\BQ$ be a probability measure on $\Omega$ given by \eqref{etailprocess},
where $\BQ'$  is a probability measure on $\Omega$ such that
$\BQ'(|Y_0|=1)=1$. Then $\xi$ is mass-stationary w.r.t.\ $\BQ$ iff
\begin{align}\label{mecketai5}
\BE_{\BQ} \int g(\vartheta_s,-s)\I\{|W_s|> 0\} \,\lambda(ds)=
\BE_{\BQ} \int g(|W_s|^{-1}\vartheta_0,s)|W_s|^\alpha \,\lambda(ds),
\end{align}
for each measurable $g\colon\bF\times\BG\to[0,\infty]$,
where $\vartheta_s\omega:=|Y_0(\omega)|^{-1}\theta_s\omega$, $\omega\in\Omega$,
whenever $|Y_0(\omega)|>0$. Such a generalization is certainly useful
when considering more randomness. For instance we may consider a second Polish
space $\BH'$ and  a suitable subset of $(\BH\times\BH')^\BG$.
The shifts are defined as before, while multiplication acts only on
the first component $Y(\omega)$ of an element $\omega\in (\BH\times\BH')^\BG$.
%some independent randomization. 
If $\BQ(\xi(\BG)=0)=0$,
Palm theory would still guarantee 
the existence of stationary measure $\BP$ (uniquely determined
on $\{\xi(\BG)>0\}$) such that $\BP_\xi=\BP(\cdot\cap\{|Y_0|>1\})=\BQ$.
\end{remark}

%\begin{remark}\label{rgen2}\rm Another generalization of the
%current setting is to allow $Y$ to take values
%in some measurable space $(\BX,\mathcal{X})$
%equipped with a measurable action $(u,z)\mapsto uz$ 
%from $(0,\infty)\times\BX$ to $\BX$ and a measurable mapping
%$|\cdot|\colon\BX\to [0,\infty)$ satisfying 
%$|uz|=u|z|$ for all $u>0$. The assumptions on $(\bF,\mathcal{F})$ 
%remain the same.
%\end{remark}

\section{Tail measures}\label{sectailmeasure}

In this section we let $(\bF,\mathcal{F})$ be as in Section \ref{sectailprocess}.
Throughout we work with the exceedance random measure $\xi$ (defined by 
\eqref{exceedancemeasure}) and the random measure $\xi'$, defined by
\eqref{xiprime}.
We say that a measure $\nu$ on $\bF$ is a {\em tail measure} if
\begin{align}\label{tail1}
&\int\I\{|Y_s|>0\}\,\lambda(ds)>0,\quad \nu\text{-a.e.},\\
\label{tail2c}
%&\nu(|Y_s|>1\}=1,\quad \lambda\text{-a.e.\ $s\in \BG$},
%\end{align} 
& \BE_\nu \xi(B)<\infty,\quad \text{$B\in\cG$ compact},
\end{align}
and if there exists an $\alpha>0$ such that $\nu$ is 
{\em $\alpha$-homogeneous}, that is
\begin{align}\label{ehomo}
\nu(u Y\in\cdot )=u^\alpha \nu(Y\in\cdot),\quad u>0.
\end{align}
In accordance with the literature we call $\alpha$ the {\em index} of $\nu$.

This definition extends the one in \cite{DoHaSo18}.
%In the stationary case this definition generalize in \cite{DoHaSo18,So21}.
A rather general (but slightly different) definition of a tail measure has
very recently been given in \cite{BlaHashShev21}.
In this paper we are mostly interested in stationary tail measures.
In this case \eqref{tail2c} implies that
$B\mapsto\BE_\nu \xi(B)$ (the intensity measure of $\xi$) is a finite multiple
of the Haar measure $\lambda$. In accordance with the
literature we shall always assume then, that this
multiple equals $1$, that is
\begin{align}\label{tail2}
&\BE_\nu \int\I\{s\in B,|Y_s|>1\}\,\lambda(ds)=\lambda(B),\quad B\in\cG,
\end{align} 
or, equivalently,
\begin{align}\label{tail2a}
\nu(|Y_0|>1)=1.
\end{align} 
Up to Remark \ref{rzerofield} our definition of a stationary
tail measure generalizes the one given \cite{So21} in the case $\BG=\R$.

If $\BQ$ is a probability measure on $\bF$ such that $\BQ(\xi(\BG)=0)=0$ and
$\xi$ is mass-stationary w.r.t.\ $\BQ$, then Theorem \ref{tMecke}
shows that there exists
a stationary measure $\nu$ on $\bF$ (uniquely determined on $\{\xi(\BG)>0\}$)
such that $\BQ=\nu_\xi$ is the Palm measure of $\xi$ w.r.t.\ $\nu$.
The main purpose of this section is to show that, if
$\BQ$ is spectrally decomposable, then $\nu$ is a tail measure.

\begin{remark}\label{rzerofield}\rm
Condition \eqref{tail1} means that
$\nu(\xi'(\BG)=0)=0$ and should be compared
with the condition $\nu(Y\equiv 0)=0$, made in
\cite{So21}. Our assumption is (slightly) stronger, also
in the stationary case. 
Without any topological structure of $\bF$ such
a stronger assumption appears to be appropriate.
(The set $\{Y\equiv 0\}$ does not even need to be measurable.)
\end{remark}

If $\nu$ is a $\sigma$-finite measure on $\bF$ and $\eta$ a random measure on $\BG$ 
we define the {\em Campbell measure} 
\begin{align*}
C_{\nu,\eta}:=\BE_\nu\int \I\{(Y,s)\in\cdot\}\,\eta(ds),
\end{align*}
which is a measure on $\bF\times\BG$.
%The next lemma shows in particular that a tail measure
%is $\sigma$-finite.
It is well-known (and easy to prove) that $C_{\nu,\xi'}$ determines
$\BP$ on the event $\{\xi'(\BG)>0\}$. For tail measures this can be
refined as follows.

\begin{lemma}\label{l1} Let $\nu$ be a tail measure on $\bF$. 
Then $\nu$ is $\sigma$-finite and uniquely determined by $C_{\nu,\xi}$.
\end{lemma}
\begin{proof} It follows from \eqref{tail1} and \eqref{tail2c} that
\begin{align*}
\BE_\nu \int \I\{s\in B,|Y_s|>c\}\,\lambda(ds)<\infty
\end{align*}
for each $c>0$ and whenever $B\subset\BG$ is compact.
Take a sequence $B_k$, $k\in \N$, of compact sets increasing towards $\BG$.
Then $\nu$ is finite on the sets
\begin{align}\label{eUk}
U_k:=\Big\{\omega\in\bF:\int \I\{s\in B_k,|\omega(s)|\ge 1/k\}\,\lambda(ds)\ge 1/k\Big\},\quad k\in\N,
\end{align}
which increase towards $\big\{\omega\in\bF:\int \I\{|\omega(s)|>0\}\,\lambda(ds)>0\big\}$.
In view of \eqref{tail1} we obtain that $\nu$ is $\sigma$-finite.
%Therefore the measure $C_{\nu,\xi'}$
%is $\sigma$-finite. Hence there exists a measurable
%$f\colon \bF\times \BG\to(0,\infty)$ such that
%$\int f\,d C_{\nu,\xi'}<\infty$. By \eqref{tail1}
%the function $g\colon\bF\to[0,\infty]$
%defined by $g(Y):=\int f(Y,s)\xi'(ds)$ is positive $\nu$-a.e. Since 
%$\BE_\nu g(Y)=\int f dC_{\nu,\xi'}<\infty$, we obtain
%that $\nu$ is $\sigma$-finite.

By homogeneity the Campbell measure $C_{\nu,\xi}$ determines
the Campbell measures
\begin{align*}
\BE_\nu \int \I\{(Y,s)\in\cdot\}\I\{|Y_s|>c\}\,\lambda(ds)
\end{align*}
for each $c>0$ and hence also $C_{\nu,\xi'}$.
Take a measurable $\tilde h\colon\bF\times\BG\to[0,\infty)$
such that $\int \tilde{h}(\omega,s)\xi'(ds)=1$,
whenever $\xi'(\omega,\BG)>0$; see Remark \ref{rH}.
Then we obtain for each measurable $g\colon\bF \to[0,\infty]$ that
\begin{align*}
\BE_\nu \I\{\xi'(\BG)>0\}g(Y)=\BE_\nu\int g(Y)\tilde{h}(Y,s)\,\xi'(ds)
=\int g(\omega)\tilde{h}(\omega,s)\,C_{\nu,\xi'}(d(\omega,s)).
\end{align*}
Since $\nu(\xi'(\BG)=0)=0$, this proves the second assertion.
\end{proof}

%\bigskip
%The following property of a tail measure is useful.
%
%\begin{lemma}\label{lfinite} Assume that $\nu$ is a tail measure.
%Then $\nu(\xi(\BG)=0)=0$.
%\end{lemma}
%%% The lemma is true, but cannot be proved at this stage.

Given a $\sigma$-finite stationary measure $\nu$ on $\bF$,
we recall that $\nu_\xi=\nu(\cdot\cap \{|Y_0|>1\})$
is the Palm measure of $\xi$ w.r.t.\ $\nu$.
If $\nu(\xi'(\BG)=0)=0$ (e.g.\ if $\nu$ is a tail measure),
then the definition \eqref{Palm} and the shift-invariance
of the event $\{\xi'=0\}$ show that
\begin{align}\label{e845}
\nu_\xi(\xi'=0)=0.
\end{align}

\begin{corollary}\label{c34}  A stationary tail measure $\nu$ is 
uniquely determined by $\nu_\xi$.
\end{corollary}
\begin{proof} Let $\nu'$ be another stationary tail measure with
$\nu_\xi=\nu'_\xi$. By the refined Campbell theorem \eqref{erefinedC}
we obtain that $C_{\nu,\xi}=C_{\nu',\xi}$.
Lemma \ref{l1} shows that $\nu=\nu'$, as asserted.
\end{proof}

\bigskip
Next we connect tail measures with
spectrally decomposable fields. 
The first part of the following proposition is
a classical result.

\begin{proposition}\label{le23}
Let $\nu$ be a stationary tail measure with index $\alpha>0$.
Then there exists a probability measure $\BQ'$ on $\bF$
such that $\BQ'(|Y_0|=1)=1$ and 
\begin{align}\label{e8765}
\nu(\cdot\cap \{|Y_0|>0\})
=\iint \I\{u\omega\in\cdot\}\I\{u>0\}\,\alpha u^{-\alpha-1}\,\BQ'(d\omega)\,du.
\end{align}
Moreover, $\xi$ is mass-stationary with respect to the probability
measure
\begin{align}\label{eBQ2}
\BQ:=\iint \I\{u\omega\in\cdot\}\I\{u>1\}\,\alpha u^{-\alpha-1}\,\BQ'(d\omega)\,du.
\end{align}
Further we have %$\BQ(\xi'=0)=0$ and 
$\nu_\xi=\BQ$.
\end{proposition}
\begin{proof} The first part follows by a classical argument;
see also \cite{EvansMol18} for a general version. 
For the convenience of the reader we give the short proof. 
Define
\begin{align*}
\BQ':=\nu(\{\omega\in\bF: |\omega(0)|>1, |\omega(0)|^{-1}\omega\in\cdot\}).
\end{align*}
By \eqref{tail2a} (and stationarity), this is a probability measure
and we have that $\BQ'(|Y_0|=1)=1$ by definition.
Take $u>0$ and $A\in\mathcal{F}$. By \eqref{ehomo},
\begin{align*}
\nu(\{\omega\in\bF: |\omega(0)|>u, |\omega(0)|^{-1}\omega\in A\}
=u^{-\alpha}\,\BQ'(A).
\end{align*}
This implies \eqref{e8765}.

To prove the second assertion we proceed similarly as in the 
first part of the proof of Theorem \ref{tMecketail}.
Let us first note, that
\begin{align}\label{erestriction}
\BQ=\nu(\cdot\cap\{|Y_0|>1\}).
\end{align}
Let $h\colon\bF\times \BG\to[0,\infty]$ be measurable
and $\varepsilon\in(0,1]$. Then
\begin{align*}
I&:=\BE_{\BQ} \int h(\theta_sY,-s)\I\{|Y_s|>\varepsilon\} \,\lambda(ds)\\
&=\BE_{\BQ'} \iint h(u\theta_sW,-s)\I\{u|W_s|>\varepsilon,u>1\}
\alpha u^{-\alpha-1}\,du \,\lambda(ds)\\
&=\varepsilon^{-\alpha}\,\BE_{\BQ'} \iint h(\varepsilon v\theta_sW,-s)
\I\{v|W_s|>1,v>\varepsilon^{-1}\}\alpha v^{-\alpha-1}\,dv \,\lambda(ds).
\end{align*}
By \eqref{e8765} and $\BQ'(|W_0|=1)=1$,
\begin{align*}
I&=\varepsilon^{-\alpha}\,
\BE_{\nu} \iint h(\varepsilon \theta_sY,-s)\I\{|Y_s|>1,|Y_0|>\varepsilon^{-1}\}
\alpha v^{-\alpha-1}\,dv \,\lambda(ds)\\
&=\varepsilon^{-\alpha}\,
\BE_{\nu} \iint h(\varepsilon Y,-s)\I\{|Y_0|>1,|Y_{-s}|>\varepsilon^{-1}\}
\alpha v^{-\alpha-1}\,dv \,\lambda(ds),
\end{align*}
where we have used stationarity, to obtain the second equality.
From here we can proceed as in the proof of  Theorem \ref{tMecketail}
to obtain \eqref{mecketail}.

%% Not needed here
%The event $A:=\{\xi'=0\}$ is translation and scale invariant.
%By \eqref{tail1} we have $\nu(A\cap\{|Y_0|>1\})=0$.
%This implies $\BQ'(A)=0$ and then in turn also $\BQ(A)=0$

The final assertion $\nu_\xi=\BQ$ follows from 
\eqref{erestriction} and Lemma \ref{lsimplePalm}.
\end{proof}

\begin{remark}\label{r1234}\rm Let $\nu$ be a tail measure.
By Proposition \ref{le23}, $\nu_\xi(W\in\cdot)$ ($=\BQ$) determines
$\nu_\xi$ and hence, by Corollary \ref{c34} also $\nu$.
\end{remark}

\bigskip

Generalizing \cite[Theorem 2.9]{DoHaSo18} (treating $\BG=\Z$) and 
\cite[Theorem 2.3]{So21}
(treating $\BG=\R$) we next provide a  construction of a stationary tail measure $\nu$,
assuming the space shift formula \eqref{mecketail} to hold
for some given probability measure $\BQ$.
This measure $\nu$ satisfies $\nu_\xi=\BQ$. 
%We continue with yet another construction of
%a stationary tail measure.
A function $h$ from $\bF$ into some space is said
to be {\em $0$-homogeneous} if $h(u\omega)=h(\omega)$
for each $\omega\in\bF$ and each $u>0$.

\begin{theorem}\label{ttailmeasure} Assume that 
$\BQ$ is a spectrally decomposable probability measure on $\bF$ such that the space shift formula
%$\BQ(\xi'(\BG)>0)=\BQ(|W_0|=1)=1$ and that
\eqref{mecketail} holds for some $\alpha>0$. Assume that $H\colon\bF\times \BG\to[0,\infty]$
is a measurable function, $0$-homogeneous in the first coordinate
and such that
\begin{align}\label{e6.5}
  \int H(\theta_{t}W,s-t)\I\{|W_{s}|>0\}\,\lambda(ds)=1,\quad \lambda\text{-a.e.\ $t$},\, \BQ\text{-a.s}.
\end{align}
Define a measure $\nu^H$ on $\bF$ by
\begin{align}\label{einverse4}
\nu^H=\BE_\BQ\iint \I\{u\theta_{-s}W\in\cdot\}H(\theta_{-s}W,s)\I\{u>0\}\alpha u^{-\alpha-1}\,du\,
\lambda(ds).
\end{align}
Then $\nu^H$ is a stationary tail measure satisfying $(\nu^H)_\xi=\BQ$.
\end{theorem}
\begin{proof} 
For the proof we generalize some of the arguments
from \cite{DoHaSo18,So21}.
The fact that $\nu^H$ is $\alpha$-homogeneous
is an immediate consequence of the definition.
Assumption \eqref{e6.5} implies that $\BQ(\xi'(\BG)=0)=0$.
Since $\{\xi'(\BG)=0\}$ is shift and scale invariant, 
we obtain again directly from the definition of $\nu^H$
that $\nu^H(\xi'(\BG)=0)=0$, that is \eqref{tail1}.

Let $f\colon\bF\times \BG\to[0,\infty]$ be measurable
and set $\rho(du):=\I\{u>0\}\alpha u^{-\alpha-1}du$.
Then
%Similarly as in the proof of Theorem \ref{tMecketail} it follows that
\begin{align*}
I:=\BE_{\nu^H}&\int f(\theta_tY,t)\,\xi(dt)=\BE_{\nu^H}\int f(\theta_tY,t)\I\{|Y_t|>1\}\,\lambda(dt)\\
&=\BE_\BQ \iiint f(u\theta_{t-s}W,t)\I\{u|W_{t-s}|>1,H(\theta_{-s}W,s)\,\rho(du)\,\lambda(ds)\,\lambda(dt)\\
&=\BE_\BQ \iiint f(v|W_{t-s}|^{-1}\theta_{t-s}W,t)|W_{t-s}|^\alpha
H(\theta_{-s}W,s)\}\I\{v>1\}\,\lambda(ds)\,\lambda(dt)\,\rho(dv),
\end{align*}
where we have used the homogeneity of $H$ and a change of variables.
By the invariance properties of Haar measure (set $r:=t-s$ in the inner integral),
\begin{align*}
I=\BE_\BQ \iiint f(v|W_{r}|^{-1}\theta_{r}W,t)|W_{r}|^\alpha
H(\theta_{r-t}W,t-r)\}\I\{v>1\}\,\lambda(dr)\,\lambda(dt)\,\rho(dv).
\end{align*}
Now we can use assumption \eqref{mecketail} (and again the
homogeneity of $H$) to obtain that
\begin{align*}
I=\BE_\BQ \iiint f(v W,t)
H(\theta_{-t}W,r+t)\}\I\{|W_r|>0\}\I\{v>1\}\,\lambda(dr)\,\lambda(dt)\,\rho(dv).
\end{align*}
By assumption \eqref{e6.5},
\begin{align}\label{e6.172}
\BE_{\nu^H}\int f(\theta_tY,t)\,\xi(dt)
=\BE_\BQ \iint f(v W,t)\I\{v>1\}\,\lambda(dt)\,\rho(dv).
\end{align}
%From here we can proceed as in the proof of Theorem \ref{t34}.
%For the convenience of the reader we provide the details.

From \eqref{e6.172} we conclude that \eqref{tail2} holds for $\nu=\nu^H$.
The right-hand side of \eqref{e6.172}
does not depend on the specific choice of $H$. Take $r\in \BG$
and apply  \eqref{e6.172} with $H^r$ instead of $H$,
where  $H^r(\omega,s):=H(\theta_{r}\omega,s-r)$. Lemma \ref{l1} yields
that $\nu^H=\nu^{H_r}$. On the other hand
we obtain for each measurable $g\colon\bF\to[0,\infty]$ that
\begin{align*}
\BE_{\nu^{H^r}}g(\theta_rY)=\BE_\BQ\iint \I\{u\theta_{r-s}W\in\cdot\}H(\theta_{r-s}W,s-r)
\I\{u>0\}\alpha u^{-\alpha-1}\,du\,
\lambda(ds),
\end{align*}
which equals $\BE_{\nu^{H}}g(Y)$. Hence $\nu^H$ is stationary
and \eqref{e6.172} shows that $(\nu^H)_\xi=\BQ$.
\end{proof}

Next we discuss some special cases of Theorem \ref{ttailmeasure}.
Given  a measurable function $G\colon \BG\to [0,\infty]$ 
we define a measurable function
%with $\int G\,d\lambda=1$. Define a measurable function
$J_G\colon\bF\to[0,\infty]$ by
\begin{align}
J_G(\omega):=\int |\omega(s)|^\alpha G(s)\,\lambda(ds),\quad \omega\in\bF.
\end{align}
and a measure $\BQ^G$ on $\bF$ by
\begin{align}\label{eBQG}
\BQ^G:=\BE_\BQ\int\I\{J_G(\theta_{-s}W)^{-1/\alpha}\theta_{-s}W\in\cdot\}G(s)\,\lambda(ds).
\end{align}

\begin{corollary}\label{t37} Let $\BQ$ satisfy the assumptions of Theorem \ref{ttailmeasure}.
%$\BQ(\xi'(\BG)=0)=0$. $\BQ(|W_0|=1)=1$. 
Let $G\colon\bF\to [0,\infty]$ be a measurable function satisfying
\begin{align}\label{e6.17}
0<\int |W_s|^\alpha G(s+r)\,\lambda(ds)<\infty,\quad \lambda\text{-a.e.\ $r$},\,\text{$\BQ$-a.s.} 
\end{align}
Then 
\begin{align}\label{eBPG}
\nu^G:=\BE_{\BQ^G}\int \I\{u Y\in\cdot,u>0\}\alpha u^{-\alpha-1}\,du
\end{align}
is a stationary tail measure satisfying $(\nu^G)_\xi=\BQ$.
\end{corollary}
\begin{proof} We wish to apply Theorem \ref{ttailmeasure} with the
function
\begin{align*}
H(\omega,t):=J_G(\omega)^{-1}|\omega(t)|^\alpha G(t).
\end{align*}
For each $t\in\BG$ we have 
\begin{align*}
  \int H(\theta_{t}W,s-t)\I\{|W_{s}|>0\}\,\lambda(ds)
&= J_G(\theta_tW)^{-1}\int |W_s|^\alpha G(s-t)\,\lambda(ds).
\end{align*}
By assumption \eqref{e6.17} this equals $1$ for $\lambda$-a.e.\ $t$.
Since $|(\theta_{-s}W)_s|=|W_0|=1$ we obtain that $\nu^H$ is given by
\begin{align*}
\BE_{\BQ}\iint \I\{v \theta_{-s}W\in\cdot,v>0\}J_G(\theta_{-s}W)^{-1}G(s)\alpha v^{-\alpha-1}\,dv\,\lambda(ds).
\end{align*}
Changing variables $u:= J_G(\theta_{-s}W)v$ yields the assertion.
\end{proof}

\begin{corollary} \label{t34}
Let $\BQ$ be a spectrally decomposable probability measure on $\bF$
such that $\BQ(\xi'(\BG)=0)=0$. % and $\BQ(|W_0|=1)=1$. 
Assume that the space shift formula
\eqref{mecketail} holds for some $\alpha>0$. Let $G\colon\bF\to (0,\infty)$
be measurable with $\int G d\lambda=1$. Define a probability measure
$\BQ^G$ by  \eqref{eBQG}. Then $\nu^G$ defined
by \eqref{eBPG} is a stationary tail measure satisfying $(\nu^G)_\xi=\BQ$.
\end{corollary}
\begin{proof} We wish to apply Corollary \ref{t37}. The first inequality
in \eqref{e6.17} follows from our assumptions $\BQ(\xi'(\BG)>0)=1$
and $G>0$. Assumption \eqref{mecketail} implies 
for each $r\in\BG$ that
\begin{align*}
\BE_\BQ \int |W_{s}|^\alpha G(s+r)\,\lambda(ds)
\le \int G(t+s)\, \lambda(ds)\le 1.
\end{align*}
Hence the second inequality in \eqref{e6.17} holds as well, proving the result.
\end{proof}

\begin{remark}\label{r1843}\rm The assumption $\BQ(\xi'(\BG)>0)=1$,
made in Corollary \ref{t34},
is a probabilistic counterpart of \eqref{tail1}.
This assumption is very natural (see Remark \ref{r5.5}) and cannot be avoided in our
general setting.
\end{remark}

\begin{remark}\rm Consider the assumptions of Theorem \ref{ttailmeasure}
and assume moreover that $\BQ(|Y_s|>0)=1$ for $\lambda$-a.e.\ $s\in\BG$. 
Then we can choose $G=\I_B$ for any $B\in\cG$
with $0<\lambda(B)<\infty$. If, for instance, $\BQ$ is discrete, then we can
take $B=\{0\}$ to obtain that $\BQ^G=\BQ$; see also \cite[Remark 2.10]{DoHaSo18}.
\end{remark}

%\begin{remark}\rm\label{rmovinga}
%Consider the assumptions of Corollary \ref{t37} and
%assume moreover that $\BQ(Z<\infty)=1$,
%where $Z:=\int |W_s|^\alpha \lambda(ds)$. Then we can choose $G\equiv 1$,
%so that $\BQ^G$ takes the form
%\begin{align*}
%\BE_\BQ\int\I\{Z^{-1/\alpha}\theta_{-s}W\in\cdot\}\,\lambda(ds).
%\end{align*}
%This is not a probability measure. Taking a measurable function $F\colon\BG\to(0,\infty)$
%with $\int F d\lambda=1$ we can define
%the probability measure
%\begin{align*}
%\BQ^F:=\BE_\BQ\int\I\{F(s)^{-1/\alpha}Z^{-1/\alpha}\theta_{-s}W\in\cdot\}F(s)\,\lambda(ds).
%\end{align*}
%Replacing $\BQ^G$ on the right-hand side of \eqref{eBPG} 
%by $\BQ^F$ does not change the outcome, that is
%\begin{align*}
%\nu^F:=\BE_{\BQ^F}\int \I\{u Y\in\cdot,u>0\}\alpha u^{-\alpha-1}\,du
%\end{align*}
%is a stationary tail measure satisfying $(\nu^F)_\xi=\BQ$.
%\end{remark}

\begin{remark}\label{rH}\rm We can follow \cite{Mecke}
to construct a function $H$ satisfying the assumptions of
Theorem \ref{ttailmeasure}. Take a measurable partition
$\{B_n:n\in\N\}$ of $\BG$ into relatively compact Borel sets.
Define $\tilde{H}\colon \bF\times \BG\to(0,\infty)$ by
\begin{align*}
\tilde{H}(\omega,s):=\sum_n2^{-n}(\xi'(\omega,B_n)+1)^{-1}\I\{s\in B_n\}.
\end{align*}
Since $\xi'(\omega)=\xi'(W(\omega))$
we have that $\tilde{H}(\omega,s)=\tilde{H}(W(\omega),s)$.
Define a random variable $S$ by $S(\omega):=\int \tilde{H}(\omega,s)\,\xi'(\omega,ds)$,
$\omega\in\Omega$.
Then $S\le 1$ and $S>0$, whenever $\xi'(\BG)>0$. 
Define a function $H$ by
$H(\omega,s):=S^{-1}(\omega)\tilde{H}(\omega,s)$.
%satisfies \eqref{e6.5}, provided that $\BQ(\xi'(\BG)>0)=1$. We skip further details.
By definition of $\xi'$,
$\tilde{H}$ and hence also $H$ is $0$-homogeneous in the first
argument. Furthermore we have for 
%each $\omega\in \bF$ with $|W_0(\omega)|=1$ and each 
$t\in \BG$ that
\begin{align*}
\int &\tilde{H}(\theta_{t}W(\omega),s-t)\I\{|W_{s}(\omega)|>0\}\,\lambda(ds)
=\int \tilde{H}(\theta_{t}W(\omega),s-t)\,\xi'(W(\omega),ds)\\
&=\sum_n 2^{-n}(\xi'(\theta_tW(\omega),B_n)+1)^{-1}\I\{s-t\in B_n\}\,\xi'(W(\omega),ds)\\
&=\sum_n 2^{-n}(\xi'(W(\omega),B_n+t)+1)^{-1}\xi'(W(\omega),B_n+t).
\end{align*}
This is positive as soon as $\xi'(\omega,\BG)>0$.
Since
$$
S(\theta_tW(\omega))=\int\tilde{H}(\theta_tW(\omega),s)\,\xi'(\theta_tW(\omega),ds)
=\int\tilde{H}(\theta_tW(\omega),s-t)\,\xi'(W(\omega),ds),
$$
we obtain \eqref{e6.5}, provided that $\BQ(\xi'(\BG)>0)=1$.
\end{remark}

%Combining Theorems \ref{ttailmeasure} and \ref{tMecketail}
%(and taking into account Remark \ref{rH})
%yields the following corollary.

\begin{corollary}\label{cMecketail} Suppose that $\BQ$ is
a spectrally decomposable probability measure on $\bF$ 
such that $\BQ(\xi'(\BG)>0)=1$. Assume that 
$\xi$ is mass-stationary w.r.t.\ $\BQ$.
Then there exists a unique  stationary tail measure $\nu$
such that $\nu_\xi=\BQ$. This tail measure is given by 
\eqref{eBPG} and (under the hypothesis \eqref{e6.5}) also by
\eqref{einverse4}
\end{corollary}
\begin{proof} Theorem \ref{tMecketail}, assumption $\BQ(\xi'(\BG)>0)=1$
and Remark \ref{rH} allow us to apply Theorem \ref{ttailmeasure}.
Combining this with Corollary \ref{c34} shows that
\eqref{einverse4} is the unique tail measure $\nu$ with $\nu_\xi=\BQ$.
By Corollary \ref{t34}, $\nu$ is also given by \eqref{eBPG}.
\end{proof}

\bigskip
By \eqref{e430}, Corollary \ref{cMecketail} (or Corollary \ref{t34})
has the following (quite natural) consequence.

\begin{corollary}\label{c154} Assume that $Y$ is spectrally
decomposable and that $\xi$ is mass-stationary w.r.t.\ $\BQ$.
If $\BQ(\xi'(\BG)=0)=0$, then $\BQ(\xi(\BG)=0)=0$. 
\end{corollary}

\begin{remark}\rm\label{r6.14} Suppose that $\nu$ is a tail measure
and write $\BQ=\nu_\xi$.  Let $g\colon\bF\times \BG\to[0,\infty]$ be measurable
and $r>0$. Then
\begin{align*}%\label{mecke7}
\BE_{\BQ} \int g(Y,s)\I\{|Y_s|> r\}\,\lambda(ds)
=\BE_{\nu} \int g(Y,s)\I\{|Y_0|>1,|Y_s|> r\}\,\lambda(ds).
\end{align*}
By homogeneity and stationarity of $\nu$ this equals
\begin{align*}
r^{-\alpha}\BE_{\nu} &\int g(rY,s)\I\{r|Y_0|>1,|Y_s|> 1\}\,\lambda(ds)\\
&=r^{-\alpha}\BE_{\nu} \int g(r\theta_{-s}Y,s)\I\{r|Y_{s}|>1,|Y_0|> 1\}\,\lambda(ds),
\end{align*}
which yields \eqref{mecke7}. In view of  Corollary \ref{cMecketail} this
completes the proof of Lemma \ref{l5.3}.
\end{remark}

%Motivated by \cite{DoHaSo18}, 
%we proceed with an alternative description of the stationary measure $\nu$. 
%A function $\tau\colon\bF\to \BG\cup\{\infty\}$ is said
%to be {\em homogeneous} if $\tau(u\omega)=\tau(\omega)$
%for each $\omega\in\bF$ and each $u>0$.

\begin{example}\rm\label{ptailmeasure2} Assume %that $\BQ(|Y_0|>0)=1$ and 
that $\BG$ is discrete, \eqref{mecketail3} holds and that $T\colon\bF\to \BG\cup\{\infty\}$
is a measurable and $0$-homogeneous mapping satisfying
\begin{align}\label{e6.7}
\sum_{s\in \BG}\I\{T(\theta_tW)=s+t,|W_{s+t}|>0\}=1,\quad 
\BQ\text{-a.s.},\,t\in \BG,
%\BQ\text{-a.s.\ on $\{\xi'(\BG)>0\}$},\,t\in \BG,
\end{align}
Then we can apply Theorem \ref{ttailmeasure} with $H(W,s)=\I\{T(W)=s\}$.
The measure \eqref{einverse4} takes the form
\begin{align}\label{einverse9}
\nu^T:=\BE_\BQ\sum_{s\in \BG}\int \I\{u\theta_{-s}W\in\cdot,T(\theta_{-s}W)=s\}\I\{u>0\}\alpha u^{-\alpha-1}\,du,
\end{align}
providing a modest generalization of \cite[Proposition 2.12]{DoHaSo18}.
%Then $\nu^\tau$ is a stationary tail measure satisfying $(\nu^\tau)_\xi=\BQ$.
Using the arguments in \cite[Section 2.4]{DoHaSo18}
(and assuming $\BQ(Y\equiv 0)=0$)
it is possible to construct a mapping  $T$ with the preceding properties.
\end{example}

\begin{remark}\rm We can extend the mapping $T$ from 
Example \ref{ptailmeasure2} to an allocation by setting
$\tau(\omega,s):=T(\theta_s\omega,0)+s$. Then the formulas
\eqref{einverse3} and \eqref{einverse9} look very similar.
The crucial difference is that the allocation in the first
formula picks a point from $\xi$ while the one
from \eqref{einverse9} picks a point from $\xi'$. This explains
the difference in the range of integration for the scaling variable $u$.
A similar remark applies to Remark \ref{pinversepalm} 
and Theorem \ref{ttailmeasure}.
\end{remark}

\section{Spectral representation}\label{sectionspectralr}

Again we establish the canonical setting of Section \ref{sectailprocess}.
Let $\nu$ be a measure on $\bF$. In accordance with
the literature we say that $\nu$ has a {\em spectral representation}, if
there exists a probability measure $\BQ^*$ on $\bF$ and
an $\alpha>0$ satisfying
\begin{align}\label{espectral}
\nu=\BE_{\BQ^*}\int \I\{u Y\in\cdot,u>0\}\alpha u^{-\alpha-1}\,du.
\end{align}
In this case we refer to $\BQ^*$ as a {\em spectral measure} of $\nu$
and to $\alpha$ as the index of $\nu$.
Our previous results will show rather quickly that any stationary tail measure
has a spectral representation.  In a sense this section is dual to the previous one.
We start with the non-probabilistic object $\nu$ and derive
the probabilistic  representation \eqref{espectral}. 

First we will state a few basic properties of a spectral representation,
to be found (in special cases) in \cite{DoHaSo18,So21}
and in the recent preprint \cite{BlaHashShev21} dealing with more
general fields. Recall that a stationary tail measure is assumed
to be normalized as in \eqref{tail2}.

\begin{proposition}\label{p823} Suppose that $\nu$ 
admits a spectral representation with spectral measure
$\BQ^*$ and index $\alpha>0$. Assume that $\BQ^*(\xi'(\BG)>0)=1$.
Then we have:
\begin{enumerate}
\item[{\rm (i)}] $\nu$ is a tail measure iff
\begin{align}\label{e999}
\BE_{\BQ^*}\int \I\{s\in B\} |Y_s|^\alpha\,\lambda(ds)<\infty,\quad \text{$B\in\cG$ compact}.
\end{align}
\item[{\rm (ii)}] Assume in addition that \eqref{e999} holds.
%$\int_B\BE_{\BQ^*}|Y_s|^\alpha\,\lambda(ds)<\infty$ for each compact $B\subset \BG$.
Then $\nu$ is stationary iff
\begin{align}\label{estathomo}
\BE_{\BQ^*}\int g(Y,s)|Y_s|^\alpha\,\lambda(ds)
=\BE_{\BQ^*}\int g(\theta_{-s}Y,s)|Y_0|^\alpha\,\lambda(ds),
\end{align}
holds for all measurable $g\colon \bF\times \BG\to[0,\infty]$
which are $0$-homogeneous in the first argument.
If these conditions hold, then $\nu$ is a stationary tail measure iff
\begin{align}\label{e638}
\BE_{\BQ^*}|Y_s|^\alpha=1,\quad \lambda\text{-a.e.\ $s\in \BG$}.
\end{align}
\end{enumerate}
\end{proposition}
\begin{proof} For the proof we generalize the arguments
in \cite{DoHaSo18} (given for $\BG=\Z$) in a straightforward manner.
Clearly $\nu$ is $\alpha$-homogeneous. By 
$\BQ^*(\xi'(\BG)>0)=1$ and \eqref{espectral}, $\nu$
satisfies property \eqref{tail1}.

(i) Let $B\in\cG$. Then
\begin{align}\label{e888}\notag
\BE_\nu \int\I\{s\in B,|Y_s|>1\}\,\lambda(ds)
&=\BE_{\BQ^*}\iint \I\{s\in B, u |Y_s|>1\}\alpha u^{-\alpha-1}\,du\,\lambda(ds)\\
&=\BE_{\BQ^*}\int \I\{s\in B\} |Y_s|^\alpha\,\lambda(ds).
\end{align}
Hence \eqref{tail2} and \eqref{e999} are equivalent.

(ii) %Under the additional assumption the first calculation in (i)
%shows that the intensity measure $\BE\xi$ is locally finite.
%
Assume that $\nu$ is stationary and take
a measurable $g\colon\bF\times\BG\to[0,\infty]$
which is $0$-homogeneous in the first argument.
Then
\begin{align*}
\BE_{\BQ^*}\int g(Y,s)|Y_s|^\alpha\,\lambda(ds)
&=\BE_{\BQ^*}\iint g(uY,s)\I\{u|Y_s|>1\}\alpha u^{-\alpha-1}\,du\,\lambda(ds)\\
&=\BE_{\nu}\int g(Y,s)\I\{|Y_s|>1\}\,\lambda(ds).
\end{align*}
By stationarity of $\nu$ this equals
\begin{align*}
\BE_{\nu}\int g(\theta_{-s}Y,s)\I\{|Y_0|>1\}\,\lambda(ds)
=\BE_{\BQ^*}\iint g(\theta_{-s}Y,s)\I\{u|Y_0|>1\}\alpha u^{-\alpha-1}\,du\,\lambda(ds).
\end{align*}
This equals the right-hand side of \eqref{estathomo}.

Assume now that \eqref{estathomo} holds.
Take a measurable $f\colon\bF\times\BG\to[0,\infty]$ and $t\in\BG$.
Then
\begin{align*}
I&:=\BE_\nu\int f(\theta_tY,s)\I\{|Y_{s+t}|>1\}\,\lambda(ds)\\
&=\BE_{\BQ^*}\iint f(u\theta_tY,s)\I\{u|Y_{s+t}|>1\}\alpha u^{-\alpha-1}\,du\,\lambda(ds)\\
&=\BE_{\BQ^*}\iint f(u\theta_tY,s-t)\I\{u|Y_{s}|>1\}\alpha u^{-\alpha-1}\,du\,\lambda(ds)\\
&=\BE_{\BQ^*}\iint f(|Y_s|^{-1}v\theta_tY,s-t)|Y_{s}|^\alpha
\I\{v>1\}\alpha v^{-\alpha-1}\,dv\,\lambda(ds).
\end{align*}
By the homogeneity of  $|Y_s|^{-1}Y$ and \eqref{estathomo},
\begin{align*}
I&=\BE_{\BQ^*}\iint f(|Y_0|^{-1}v\theta_{t-s}Y,s-t)|Y_0|^\alpha
\I\{v>1\}\alpha v^{-\alpha-1}\,dv\,\lambda(ds)\\
&=\BE_{\BQ^*}\iint f(u\theta_{t-s}Y,s-t)
\I\{u|Y_0|>1\}\alpha u^{-\alpha-1}\,du\,\lambda(ds)\\
&=\BE_{\BQ^*}\iint f(u\theta_{-s}Y,s)
\I\{u|Y_0|>1\}\alpha u^{-\alpha-1}\,du\,\lambda(ds).  %\\
%&=\BE_{\nu}\int f(\theta_{-s}Y,s)\I\{|Y_0|>1\}\,\lambda(ds).
\end{align*}
This shows that the Campbell measures $C_{\nu\circ{\theta_t},\xi'}$ do not depend on $t\in\BG$.
But $\nu\circ{\theta_t}$ does also satisfy \eqref{tail1} and 
\eqref{ehomo} along with the assumption of (ii).
Hence Lemma \ref{l1} implies that $\nu$ is stationary.

The final assertion follows from \eqref{e888}.
\end{proof}

Let us mention the following fact; cf.\ \cite[(2.8)]{So21}.

\begin{corollary} Suppose that $\nu$ is a stationary tail measure
with index $\alpha>0$. %Let $\BQ$ be as in  Lemma \ref{le23} and 
Let $\BQ^*$ be a spectral measure of $\nu$. 
Then 
\begin{align*}
\nu_\xi(W\in\cdot)=\BE_{\BQ^*}\I\{|Y_s|^{-1}\theta_sY\in\cdot\}|Y_s|^\alpha,\quad \lambda\text{-a.e.\ $s$}.
\end{align*}
\end{corollary}
\begin{proof}
In the proof of Proposition \ref{p823} we have seen that
\begin{align*}
\BE_{\BQ^*}\int g(Y,s)|Y_s|^\alpha\,\lambda(ds)
=\BE_{\nu}\int g(\theta_{-s}Y,s)\I\{|Y_0|>1\}\,\lambda(ds).
\end{align*}
holds, provided that $g$ is 0-homogeneous in the first
argument.
The right-hand side equals $\BE_{\nu_\xi}\int g(\theta_{-s}Y,s)\lambda(ds)$.
Equivalently,
\begin{align*}
\BE_{\BQ^*}\int g(\theta_sY,s)|Y_s|^\alpha\,\lambda(ds).
=\BE_{\nu_\xi}\int g(Y,s)\,\lambda(ds)
\end{align*}
Applying this with $g(Y,s):=h(|Y_0|^{-1}Y,s)$ for a
measurable $h\colon\bF\times\BG\to[0,\infty]$ yields the assertion.
\end{proof}

\bigskip
The following result extends the stationary case of 
\cite[Theorem 2.4]{DoHaSo18}  (covering the case $\BG=\Z$)
and \cite[Theorem 2.3]{So21} (dealing with the case
$\BG=\R$). A general non-stationary
(and therefore less specific) version can be found
as Lemma 3.10 in the recent preprint \cite{BlaHashShev21}.

\begin{theorem}\label{tspecrepresentation} Suppose 
that $\nu$ is a stationary tail measure
with index $\alpha>0$.   %such that $\nu(\xi(\BG)=0)=0$. 
Then $\nu$ has a spectral
representation with a spectral measure $\BQ^*$
satisfying %$\BQ^*(\xi'(\BG)>0)=1$ and 
\eqref{estathomo} and \eqref{e638}.
\end{theorem}
\begin{proof}  By Theorem \ref{tMecke} the measure $\BQ:=\nu_\xi$ is mass-stationary.
By Proposition \ref{le23}, Theorem \ref{tMecketail} and \eqref{e845},
$\BQ$ satisfies the assumptions of Theorem \ref{ttailmeasure}.
By Corollary \ref{t34} we can therefore
define a tail measure $\nu'$ (with index $\alpha$) 
by \eqref{eBPG} for some given function $G$
with the required properties. 
Then $\nu'$ admits a spectral representation with spectral
measure $\BQ^*:=\BQ^G$.
By Corollary \ref{t34} we also have $\nu'_\xi=\BQ$,
that is $\nu'_\xi=\nu_\xi$.
%while Lemma \ref{le23} shows that $\nu_\xi=\BQ$.
%By the refined Campbell theorem \eqref{erefinedC}
%we hence obtain that $C_{\nu,\xi}=C_{\nu',\xi}$.
%Therefore $\nu$ and $\nu'$ coincide on $\{\xi(\BG)>0\}$.
%By homogeneity $\nu$ and $\nu'$ coincide for
%each $\varepsilon>0$ on $\{\xi_\varepsilon(\BG)>0\}$,
%where $\xi_\varepsilon:=\int\I\{|Y_s|>\varepsilon\}\lambda(ds)$.
%Hence $\nu$ and $\nu'$ coincide on $\{\xi'(\BG)>0\}$.
%Since \eqref{tail1} applies to both measures we obtain
%Lemma \ref{l1} shows 
Corollary \ref{c34} shows
that $\nu=\nu'$, proving the spectral representation \eqref{espectral}.
By Proposition \ref{p823}, $\BQ^*$ satisfies
\eqref{e638} and \eqref{estathomo}.
\end{proof}

\begin{remark}\label{r6.34}\rm 
The existence of a spectral representation
of a tail measure $\nu$ can also be derived from Proposition 2.8 in \cite{EvansMol18}.
Indeed, the sets $U_k$ defined in \eqref{eUk} satisfy the assumptions
of that proposition. However, Theorem \ref{tspecrepresentation}
and its proof provide more detailed information on the spectral
measure $\BQ^*$. In fact, $\BQ^*$ is explicitly given in terms
of the Palm measure $\nu_\xi$ of $\xi$ w.r.t.\ $\nu$.
\end{remark}

\bigskip 
A spectral measure is not uniquely determined by
the tail measure.
Depending on the properties of $\nu_\xi$, the proof of 
Theorem \ref{tspecrepresentation} provides several ways
of constructing a spectral measure.
%%% one can use independent random variables R>0 with E R^\alpha<\infty 
%%% to obtain new spectral measures 
The recent preprint \cite{Hashorva21} contains a systematic discussion
of the relationships between random fields 
(on $\R^d$ or $\Z^d$) satisfying \eqref{estathomo} 
and stationary tail measures.

\begin{remark}\rm Let $\nu$ be a stationary tail measure.
Then $\nu_\xi$ %$\BQ=\nu(\cdot\cap\{|Y_0|>1\})$ (given by \eqref{eBQ2})
is said to be the distribution of the {\em tail process} associated with
$\nu$; see \cite{DoHaSo18,So21}. 
Under $\nu_\xi$ the process $W$ is called
a {\em spectral (tail) process} associated with $\nu$; see again \cite{DoHaSo18,So21}.
By Corollary \ref{c34},
$\nu$ is uniquely determined by $\nu_\xi(W\in\cdot)$.
But in general, $\nu_\xi(W\in\cdot)$ is not a spectral measure of $\nu$.
This clash of terminology is a bit unfortunate.
%Therefore it might be better to use the terminology ``spectral process'' 
%only for a limit object, just as in the seminal paper \cite{BaSe09}. 
\end{remark}

A tail measure $\nu$ is said to admit a {\em moving shift representation}
if there exists a probability measure $\BQ^*$ on $\bF$ such that
\begin{align}\label{emovingshift}
\nu=\BE_{\BQ^*}\iint \I\{u \theta_{s}Y\in\cdot,u>0\}\alpha u^{-\alpha-1}\,du\,\lambda(ds).
\end{align}

\begin{theorem}\label{tspecmovingshift} 
Suppose that $\nu$ is a stationary tail measure
with index $\alpha>0$. Then there exists a probability
measure $\BQ^*$ on $\bF$ such that \eqref{emovingshift}
holds iff
\begin{align}\label{e7.4}
\int |Y_s|^\alpha \lambda(ds)<\infty,\quad  \nu\text{-a.e.}
\end{align}
\end{theorem}
\begin{proof} 
Assume first, that \eqref{e7.4} holds. 
As noticed in the proof of Theorem \ref{tspecrepresentation}
the probability measure $\BQ:=\nu_\xi$ satisfies the assumptions
of Theorem \ref{ttailmeasure}.  Define the probability measure
\begin{align}\label{eBQ*}
\BQ^*:=\BQ(Z^{-1/\alpha}W\in\cdot),
\end{align}
where $Z:=\int |W_s|^\alpha \lambda(ds)$.
Applying Corollary \ref{t37} with $G\equiv 1$ shows the right-hand side of \eqref{emovingshift}
is a stationary tail measure $\nu'$ with $\nu'_\xi=\BQ$.
As in the proof of Theorem \ref{tspecrepresentation} we obtain $\nu=\nu'$.

Assume, conversely, that \eqref{emovingshift} holds.
Then
\begin{align*}
1=\nu(|Y_0|>1)=\BE_{\BQ^*}\iint \I\{u|Y_s|>1,u>0\}\alpha u^{-\alpha-1}\,du\,\lambda(ds)
=\BE_{\BQ^*}\int |Y_s|^\alpha\,\lambda(ds).
\end{align*}
Hence $\BQ^*(A)=0$, where $A:=\{\int |Y_s|^\alpha \lambda(ds)=\infty\}$.
Since $A$ is invariant under translation and scaling, we obtain from
\eqref{emovingshift} that $\nu(A)=0$.
\end{proof}

We refer the reader to \cite{DoHaSo18,DoKab17,So21}
for a more detailed analysis of moving shift representations
for special groups $\BG$ and under additional continuity assumptions
on $Y$. Extending some of those results to general groups
is an interesting task, beyond the scope of this paper.

\section{Anchoring maps}\label{secanchor}

In this section we let $Y$ and $\xi$ be as in Section \ref{secexceedance}
and suppose that $\BQ$ is a probability measure on $(\Omega,\mathcal{A})$
such that $\xi$ is mass-stationary w.r.t.\ $\BQ$.
 
%Recall that the support $\supp\mu$ of a locally finite measure $\mu$ on
%$\BG$ is the smallest closed subset of $\BG$ whose complement has
%vanishing $\mu$-measure. It is not hard to prove
%that $(\omega,s)\mapsto \I\{s\in\supp\xi(\omega)\}$ is measurable. 
Following \cite{Plan21,So21} we say that a measurable mapping
$T\colon\bF\to\BG$ is   an {\em anchoring map} if
\begin{align}
T(\theta_s\omega)=T(\omega)-s,\quad s\in\BG,\,\text{if $0<\xi(\omega,\BG)<\infty$}.
\end{align}
In stochastic geometry such functions are known as {\em center functions};
see e.g.\ \cite[Chapter 17]{LastPenrose17}.

In the following the number
\begin{align}\label{eindex}
\vartheta:=\BE_\BQ \xi(\BG)^{-1}.
\end{align}
will play an important rule. If $\BQ(\xi(\BG)<\infty)>0$, then $\vartheta>0$.

\begin{proposition}\label{p6.3} Assume that $\xi$ is mass-stationary w.r.t.\ $\BQ$.
Assume also that
\begin{align}\label{e8.8}
\BQ(0<\xi(\BG)<\infty)=1
\end{align}
and $\vartheta<\infty$.
Let $T$ be an anchoring map and define the probabiliy measure
\begin{align}\label{palm2}
\BQ_T:=\vartheta^{-1}\BE_\BQ\xi(\BG)^{-1}\I\{\theta_{T}\in\cdot\}
\end{align}
Then we have for all
measurable $g\colon\Omega\to[0,\infty]$ that
\begin{align}\label{palm1}
\BE_\BQ g&=\vartheta\, \BE_{\BQ_T}\int g\circ\theta_s\,\xi(ds).
\end{align}
\end{proposition}
\begin{proof} Let $\nu$ be the $\sigma$-finite
stationary measure on $\Omega$ such that $\BQ=\nu_\xi$ and $\nu(\xi(\BG)=0)=0$.
Define an allocation $\tau$ by $\tau(\omega,s):=T(\theta_s\omega)+s$.
By assumption $\tau(\omega,s)=T(\omega)$ for each $s\in\BG$, provided that
$0<\xi(\omega,\BG)<\infty$. Moreover,
$$
\eta:=\I\{0<\xi(\omega,\BG)<\infty\}\I\{T\in\cdot\}
$$
is an invariant simple point process. 
By Proposition \ref{palloc} we have for each measurable $h\colon\Omega\times\Omega\to[0,\infty]$
that
\begin{align*}
\BE_\BQ h(\theta_0,\theta_{\tau(0)})
=\BE_{\nu_\eta} \int \I\{\tau(s)=0\}h(\theta_s,\theta_0)\,\xi(ds).
\end{align*}
Since $\nu_\eta(\xi(\BG)\in\{0,\infty\})=0$, this means
\begin{align*}
\BE_\BQ h(\theta_0,\theta_{T})
=\BE_{\nu_\eta} \int \I\{T=0\}h(\theta_s,\theta_0)\,\xi(ds).
\end{align*}
It follows straight from the definition \eqref{Palm} 
that $\nu_\eta(T\ne 0)=0$.
Therefore
\begin{align}\label{e6.4}
\BE_\BQ h(\theta_0,\theta_{T})
=\BE_{\nu_\eta} \int h(\theta_s,\theta_0)\,\xi(ds).
\end{align}
Applying this
to the function $(\omega,\omega')\mapsto \xi(\omega',\BG)^{-1}g(\omega')$
(for some measurable $g\colon\Omega\to[0,\infty]$) yields
\begin{align*}
\BE_\BQ \xi(\BG)^{-1}g(\theta_{T})=\BE_{\nu_\eta} g
\end{align*}
and therefore $\BQ_T=\vartheta^{-1}\nu_\eta$. Hence \eqref{palm1}
follows from \eqref{e6.4}.
\end{proof}

\begin{example}\rm
Assume that $\BG$ is discrete as in Example \ref{ex2.3},
so that $\xi$ is a simple point process on $\BG$.
Under assumption \eqref{e8.8} we have $\vartheta\le 1$
with equality iff $\BQ(\xi(\BG)=1)=1$.
From \eqref{palm1} we obtain  
for each measurable $f\colon\bF\to[0,\infty]$ that
\begin{align*}
\BE_\BQ f\I\{T=0\}&=\vartheta\,\BE_{\BQ_T}\int f\circ\theta_s\I\{T\circ \theta_s=0\}\,\xi(ds)\\
&=\vartheta\,\BE_{\BQ_T} \int f\circ\theta_s\I\{T=s\}\,\xi(ds)
=\BE_{\BQ_T} f \I\{|Y_0|>1\},
\end{align*}
where we have used that $\BQ_T(T=0)=1$. As in \cite{Plan21}
it is natural to assume that
\begin{align}\label{einthemass}
|Y_T|>1.
\end{align}
Then $\BE_\BQ f\I\{T=0\}=\BE_{\BQ_T} f $.
Hence $\vartheta=\BQ(T=0)$ and 
$\BQ_T=\BQ(\cdot \mid T=0)$. 
Therefore Proposition \ref{p6.3} extends \cite[Proposition 3.2]{Plan21}.
Further formulas in the spectrally decomposable case can be found in Subsection 3.3 of 
\cite{Plan21}. 
\end{example}

In the remainder of this section we establish the canonical setting
of Section \ref{sectailprocess}. We assume that $\BQ$ is spectrally
decomposable and satisfies $\BQ(\xi'(\BG)=0)=0$.
By Corollary \ref{t34} we can associate with $\BQ$ a unique
stationary tail measure $\nu$ such that $\nu_\xi=\BQ$.
We assume moreover that
\begin{align}\label{e7.4a}%\label{e7.4}
\int |Y_s|^\alpha \lambda(ds)<\infty,\quad  \BQ\text{-a.e.}
\end{align}
Note that
\begin{align*}
\xi(\BG)\le \int\I\{|Y_s|>1\}|Y_s|^\alpha\,\lambda(ds)\le \int |Y_s|^\alpha\,\lambda(ds), 
\end{align*}
so that $\BQ(0<\xi(\BG)<\infty)=1$; see also Corollary \ref{c154}.
Since \eqref{e7.4a} does also hold $\nu$-a.e., we can 
and will use the moving shift representation \eqref{emovingshift}
with $\BQ^*$ given by \eqref{eBQ*}.

For each $\omega\in\bF$ the function
$u\mapsto \int \I\{|\omega(s)|>u^{-1}\}\,\lambda(ds)$
from $(0,\infty)$ to $[0,\infty]$ is increasing and left-continuous.
Therefore we can define a random variable $\kappa$ by
\begin{align}\label{ekappa}
\kappa:=\inf\{u>0:\xi(uZ^{-1/\alpha} W,\BG)>0\}.
\end{align}
The following lemma gives an alternative expression
for the number $\vartheta$ defined by \eqref{eindex}.

\begin{lemma}\label{l8.3} Assume that $\BQ$ is spectrally decomposable, that $\xi$ is mass-stationary
w.r.t.\ $\BQ$ and that  \eqref{e7.4a} holds. Then 
$\vartheta=\BE_\BQ\kappa^\alpha$.
\end{lemma}
\begin{proof} We start with a preliminary comment. Since $\BQ(|Y_0|>0)=1$ we have $\BQ$-a.e.
$Z^{-1/\alpha} W=\tilde{Z}^{-1/\alpha} Y$, where
\begin{align*}
\tilde Z:=\int |Y_s|^\alpha\,\lambda(ds).
\end{align*}
In particular,
\begin{align}\label{ekappa2}
\kappa=\inf\{u>0:\xi(u\tilde{Z}^{-1/\alpha} Y,\BG)>0\}.
\end{align}
By our assumptions, $\BQ(0<\tilde Z<\infty)=1$.
Since for each $v>0$
\begin{align*}
\int\I\{v|Y_s|>1\}\,\lambda(ds)\le v^\alpha \int |Y_s|^\alpha\,\lambda(ds), 
\end{align*}
we obtain
\begin{align}\label{e8.9}
\xi(u\tilde{Z}^{-1/\alpha} Y,\BG)<\infty,\quad u>0,\, \BQ\text{-a.e.}
\end{align}

Since $\BQ$ is the Palm measure of $\xi$ w.r.t.\ $\nu$
we have $\vartheta=\BE_\nu \I\{|Y_0|>1\}\xi(\BG)^{-1}$.
Hence we obtain from \eqref{emovingshift} and \eqref{eBQ*} that
\begin{align*}
\vartheta=\BE_\BQ\iint \I\{u\tilde{Z}^{-1/\alpha}|Y_s|>1\}
\xi(u\tilde{Z}^{-1/\alpha}Y,\BG)^{-1}\alpha u^{-\alpha-1}\,du\,\lambda(ds).
\end{align*}
By  Fubini's theorem and  \eqref{e8.9},
\begin{align*}
\vartheta=\BE_\BQ\int \I\{\xi(u\tilde{Z}^{-1/\alpha}Y,\BG)>0\}\alpha u^{-\alpha-1}\,du.
\end{align*}
By \eqref{ekappa2} this yields
$\vartheta=\BE_\BQ\int  \I\{u>\kappa\}\alpha u^{-\alpha-1}\,du$
and hence the asserted formula.
\end{proof}

Let $\tau$ be an allocation such that 
\begin{align}\label{eallocanchor}
\tau(u\omega,s)=\tau(u\omega,0)\in\BG,\quad \BQ\otimes\lambda\otimes\lambda_+\text{-a.e. $(\omega,s,u)$},
\end{align}
where $\lambda_+$ denotes Lebesgue measure on $(0,\infty)$.
In view of \eqref{e8.9} we can interpret $\tau(0)$ as an almost every version
of an anchoring map.
Motivated by \cite[Section 2.3]{So21} we collect some (preliminary) information
on the distribution of $(\tau(Y,0),Y)$.
Though the principal calculations are similar, we cannot use the specific
moving shift representation from \cite[Theorem 2.9]{So21}.

\begin{lemma}\label{l8.4} Assume that the assumptions of Lemma \ref{l8.3} hold.
Let $\tau$ be an allocation satisfying \eqref{eallocanchor} and
suppose that $g\colon\bF\times\BG\to [0,\infty]$
is measurable and shift-invariant in the first coordinate.  Then
\begin{align}\label{e8.7}
\BE_\BQ g(Y,\tau(0))=\BE_{\BQ^*}\iint \I\{u>\kappa\}\I\{u|Y_{s+\tau(uY,0)}|>1\}g(u Y,-s)
\alpha u^{-\alpha-1}\,du\,\lambda(ds).
\end{align}
In particular $\BQ(\tau(0)\in\cdot)$ has the $\lambda$-density
\begin{align}
f_\tau(s):=\BE_{\BQ^*}\int \I\{u>\kappa,u|Y_{-s+\tau(uY,0)}|>1\}\alpha u^{-\alpha-1}\,du,\quad s\in\BG.
\end{align}
\end{lemma}
\begin{proof}
We have
\begin{align*} 
\BE_\BQ g(Y,\tau(0))
&=\BE_\nu\I\{|Y_0|>1\}g(Y,\tau(0))\\
&=\BE_{\BQ^*}\iint \I\{u|Y_s|>1\}g(u \theta_sY,\tau(u\theta_sY,0))
\alpha u^{-\alpha-1}\,du\,\lambda(ds)\\
&=\BE_{\BQ^*}\iint \I\{u|Y_s|>1\}g(u Y,\tau(uY,0)-s)
\alpha u^{-\alpha-1}\,du\,\lambda(ds),
\end{align*}
where we have used that $\tau(u\theta_sY,0)=\tau(uY,s)-s=\tau(uY,0)-s$
holds for $\lambda\otimes\lambda_+$-a.e.\ $(s,u)$ and  $\BQ^*$-a.e. Changing variables gives
\begin{align*}
\BE_\BQ g(Y,\tau(0))
=\BE_{\BQ^*}\iint \I\{u|Y_{s+\tau(uY,0)}|>1\}g(u Y,-s)
\alpha u^{-\alpha-1}\,du\,\lambda(ds).
\end{align*}
If $u<\kappa$ then $\int \I\{u|Y_{s+\tau(uY,0)}|>1\}\,\lambda(ds)=0$.
Therefore \eqref{e8.7} follows.
The second assertion is an immediate consequence.
\end{proof}

In the remainder of the section we assume that $\tau$ is an allocation
satisfying \eqref{eallocanchor}.
Suppose that $h\colon\bF\to[0,\infty)$ is measurable and shift invariant. Then
we obtain from \eqref{e8.7} that
\begin{align}
\BE[h(Y)\mid \tau(0)=s]
=f_\tau(s)^{-1}\BE_{\BQ^*}\int \I\{u>\kappa\}\I\{u|Y_{s+\tau(uY,0)}|>1\}h(uY)
\alpha u^{-\alpha-1}\,du
\end{align}
holds for $\BQ(\tau(0)\in\cdot)$-a.e.\ $s$.
To discuss this formula we make the ad hoc assumption
\begin{align}
\lim_{s\to 0}\I\{u|Y_{s+\tau(uY,0)}|>1\}=1,\quad u>\kappa,\,\lambda_+\text{-a.e.\ $u$},\,\BQ\text{-a.e.},
\end{align}
see \cite[(2.25)]{So21} for a similar hypothesis in the case $\BG=\R$.
This can be seen as a continuous space version of \eqref{einthemass}
and might be achieved under appropriate
continuity assumptions on $Y$.
If, in addition, $\vartheta=\BE\kappa^\alpha<\infty$, then
dominated convergence yields the existence of the limit
\begin{align}
\lim_{s\to 0}f_\tau(s)=\vartheta.
\end{align}
If 
\begin{align*}
\BE_{\BQ^*}\int \I\{u>\kappa\}h(uY) \alpha u^{-\alpha-1}\,du<\infty
\end{align*}
then dominated convergence yields the existence of the limit
\begin{align}
\lim_{s\to 0}\BE[h(Y)\mid \tau(0)=s]
=\vartheta^{-1}\BE_{\BQ^*}\int \I\{u>\kappa\}h(uY)\alpha u^{-\alpha-1}\,du.
\end{align}
In particular we may take $h(Y)=\xi(\BG)$. Indeed, we have that
\begin{align*}
\BE_{\BQ^*}\int &\I\{u>\kappa\}\xi(uY,\BG) \alpha u^{-\alpha-1}\,du
=\BE_{\BQ^*}\iint \I\{u|Y_s|>1\} \alpha u^{-\alpha-1}\,du\,\lambda(ds)\\
&=\BE_{\BQ^*}\int |Y_s|^\alpha\,\lambda(ds)
=\BE_{\BQ}\int  Z^{-1} |W_s|^\alpha\,\lambda(ds)=1.
\end{align*}
Therefore,
\begin{align}\label{extremalindex}
\lim_{s\to 0}\BE[\xi(\BG)\mid \tau(0)=s]=\vartheta^{-1};
\end{align}
see \cite[(2.26)]{So21} for the case $\BG=\R$.
In view of the discussion in  \cite{KulSo2020,So21}
we might call $\vartheta$ the {\em candidate extremal index} of $Y$.

The results of this section are certainly preliminary. But without continuity 
assumptions on the elements of $\bF$ it seems difficult to make
further progress. If $\BG=\R^d$ and $\bF$ is a Skorohod space 
(see Example \ref{ex2.2}), then it might be possible to 
establish an analog of \cite[Theorem 2.9]{So21}. 
In particular the assumptions of Lemma \ref{l8.3} should
then imply $\vartheta<\infty$.

\section{Concluding remarks}\label{secconcludingr}

The results from Sections \ref{sectailprocess}-\ref{sectionspectralr}
generalize %immediately to the setting described in Remark \ref{rgen2}.
%They also generalize 
to the setting
described in Remark \ref{rgen1}. This would mean, for instance, that
tail measures are then defined on a more general space $\Omega$
and not just on the function space $\bF$.
To avoid an abstract (and potentially confusing)
notation we have chosen to stick to the present more
specific setting.

Given the results of this paper, one might define a tail process
in an intrinsic way, namely as a spectrally decomposable random
field $Y=(Y_s)_{s\in \BG}$ such that the exceedance random
measure is mass-stationary. It would be interesting to identify
such processes as the tail processes of regularly varying stationary fields,
beyond the known special cases. It would also be interesting to
further study tail measures of such fields, as introduced in great generality
in \cite{SamoOwada12}. In particular it might be  worthwhile exploring
further relationships between tail measures, tail processes and
Palm calculus.

\bigskip
\noindent
{\bf Acknowledgments:} This work arose from the workshop
``Regular Variation and Related Themes'' held at the IUC in Dubrovnik
in November 2021.
The author is very grateful to the organizers Bojan Basrak and Ilya Molchanov
for making this event possible.
The  author is very grateful to the referees whose insightful comments, proposals and
questions have helped to improve the paper.

%\bigskip
%\noindent
%{\bf Data availability statement:} 
%Data sharing not applicable to this article as no datasets were generated or analyzed during 
%the current study.

%\bigskip
%\noindent
%{\bf Conflict of Interest:}
%The author declares that he has no conflict of interest.

\end{document}